\documentclass[a4paper,11pt]{amsart}

\usepackage{latexsym}
\usepackage{a4wide}
\usepackage{amscd}
\usepackage{graphics}
\usepackage{amsmath}
\usepackage{amsthm}
\usepackage{amssymb}
\usepackage{bm}
\usepackage{mathrsfs}
\usepackage{enumerate}
\usepackage{hyperref}
\usepackage{pgf,tikz}
\usepackage{dsfont}
\usepackage{soul}

\usepackage[utf8]{inputenc}
\usepackage{lmodern}

\numberwithin{equation}{section}

\newtheorem{theorem}{Theorem}[section]
\newtheorem{corollary}[theorem]{Corollary}
\newtheorem{lemma}[theorem]{Lemma}
\newtheorem{proposition}[theorem]{Proposition}
\theoremstyle{definition}
\newtheorem{definition}[theorem]{Definition}

\newtheorem{remark}[theorem]{Remark}

\newcommand{\R}{\mathbb{R}}	% Real numbers
\newcommand{\N}{\mathbb{N}} % Natural numbers

\renewcommand{\phi}{\varphi}
\newcommand{\dx}{\,\mathrm{d}x}	% dx
\newcommand{\dxdy}{\,\mathrm{d}x \mathrm{d}y}
\newcommand{\dsi}{\,\mathrm{d}\sigma}	% dsigma
\newcommand{\ds}{\,\mathrm{d}S}	% ds
\renewcommand{\d}{\,\mathrm{d}}
\newcommand{\nnu}{\bm{\nu}}	% Bold Exterior Unit
\newcommand{\norm}[1]{\left\lVert #1 \right\lVert}
\newcommand{\abs}[1]{| #1 |}
\newcommand{\sub}{\subseteq}

\DeclareMathOperator{\dive}{\mathrm{div}}

\newenvironment{bvp}{\left\{\begin{aligned}  }{\end{aligned}\right.}

\title{On a weighted two-phase boundary obstacle problem}
\author{Donatella Danielli}
\address{School of Mathematical and Statistical Sciences, Arizona State University, Tempe, AZ 85281, USA}
\email{DDanielli@asu.edu}

\author{Roberto Ognibene}
\address{Dipartimento di Ingegneria meccanica, energetica, gestionale e dei trasporti,
	Università degli studi di Genova, 16145 Genova, Italy}
\email{roberto.ognibene@edu.unige.it}

\begin{document}

\begin{abstract}
	In this work we consider an inhomogeneous two-phase obstacle-type problem driven by the fractional Laplacian. In particular, making use of the Caffarelli-Silvestre extension, Almgren and Monneau type monotonicity formulas and blow-up analysis, we provide a classification of the possible vanishing orders, which implies the strong unique continuation property. Moreover, we prove a stratification result for the nodal set, together with estimates on its Hausdorff dimensions, for both the regular and the singular part. The main tools come from geometric measure theory and amount to Whitney's Extension Theorem and Federer's Reduction Principle.
\end{abstract}

\maketitle


\medskip


\section{Introduction}

In the present paper we investigate local properties of solutions of a stationary two-phase fractional problem, with particular emphasis on regularity aspects and on the structure of the nodal set. Our investigation focuses on the study of weak solutions of the following boundary value problem
\begin{equation}\label{eq:strong_form_1}
	\begin{bvp}
		-\dive(y^a\nabla u)&=0, &&\text{in } {B_1^+}, \\
		-\lim_{y\to 0^+}y^a\frac{\partial u}{\partial y}&=\lambda_-((u-h)^-)^{p-1}-\lambda_+((u-h)^+)^{p-1},&& \text{on } {B_1'}.
	\end{bvp}
\end{equation}
Here $a\in(-1,1)$, $t^+:=\max\{t,0\}$ and $t^-:=\max\{-t,0\}$ denote the positive and negative parts of $t$, $\lambda_-,\lambda_+\geq 0$ are nonnegative constants, $p\geq 2$ and, for $r>0$,
\begin{equation*}
B_r^+:=\{z=(x,y)\in\R^{N+1}_+\colon \abs{z}<r \},\quad B_r':=\{x\in \R^N\colon \abs{x}<r \},
\end{equation*}
 with
\[
	\R^{N+1}_+:=\{(x,y)\colon x\in\R^N,~y>0 \}=\R^N\times(0,+\infty),
\]
for $N\geq 1$. We will refer to the function $h \colon {B_1'}\to\R$ as the   \emph{(thin) obstacle}. Thanks to the extension procedure established in \cite{Caffarelli2007}, one can see that solutions to \eqref{eq:strong_form_1} are closely related to solutions of an equation in $N$ dimensions, driven by the fractional Laplacian
$$
{(-\Delta)^s u(x_0):=C(N,s) \lim_{\varepsilon\to 0}\int_{\mathbb{R}^N\setminus B_\varepsilon'(x_0)}\frac{u(x_0)-u(x)}{|x_0-x|^{N+2s}}\ dx},
$$
where $s=(1-a)/2\in (0,1)$ and $C(N,s)>0$ is a constant, while
\[
	B_r'(x_0):=\{x\in\R^N\colon \abs{x-x_0}<r\}.	
\]
Loosely speaking, it is shown in \cite{Caffarelli2007} that, in a suitable sense,
 $$
 (-\Delta)^{s}u(x_0)=-\kappa(N,a) \lim_{y\to 0^+}y^a\frac{\partial u}{\partial y} (x_0,y),
 $$
 for a {certain (explicit)} constant $\kappa(N,a)>0$. Therefore, one can think of the trace on $\{y=0\}$ of  $u$ in \eqref{eq:strong_form_1} as a solution to
\begin{equation}\label{eq:strong_fractional}
	(-\Delta)^s u=\lambda_-((u-h)^-)^{p-1}-\lambda_+((u-h)^+)^{p-1}.
\end{equation}
 Among many others, we refer for instance to \cite{Dinezza2012} and \cite{MR3916700} for a more thorough description of the operator $(-\Delta)^s$, the fractional setting, and the extension procedure (see also \cite{Abatangelo20191,Bucur2016,Chen2020}).  Thanks to this  connection, the problem we study can be read in two different ways. On the one hand, it can be seen as a two-phase penalized obstacle problem for the fractional Laplacian on $\R^N$, as in \eqref{eq:strong_fractional}. On the other hand, in its extended formulation \eqref{eq:strong_form_1}, one can understand it as a (weighted) two-phase boundary penalized obstacle problem associated to the operator $-\dive(y^a\nabla (\cdot))$.

Although we will  keep in mind the parallel just outlined, in this work we will focus on the study of the extended problem \eqref{eq:strong_form_1}. More precisely, our main objectives are to investigate the regularity of the solution and the structure of
the \emph{free boundary}
\[
\partial\{ u(\cdot,0)\neq 0 \}\cap B_1'.
\]
We now observe a couple of limit situations that may emerge from \eqref{eq:strong_form_1}. If we consider both $\lambda_-=\lambda_+=0$, then the boundary condition in \eqref{eq:strong_form_1} becomes of homogeneous Neumann type, while if we let $\lambda_+\to +\infty$ and $\lambda_-=0$, then the condition becomes
	\[
	(u-h)^+y^a\frac{\partial u}{\partial y}=0\quad\text{on }B_1'.
	\]
	In the case $s=1/2$, this is known as the \emph{thin obstacle problem} (or even \emph{Signorini problem}, see \cite{Signorini1959}).
	
Boundary value problems such as \eqref{eq:strong_form_1} find many applications as physical models. We briefly describe an example. Let $\Omega\sub\R^{N+1}$ be a continuum, conductive body, and let us consider the problem of regulating the temperature of its boundary $\partial\Omega$. In particular, let us assume we have a reference temperature $h\colon \partial\Omega\to \R$, and we want the temperature $u(z,t)$, for $(z,t)\in\partial\Omega\times(0,+\infty)$, to stay as close as possible to $h(z)$ at any time $t>0$. The mathematical formulation of this phenomenon is given by the following boundary value problem
\[
	\begin{bvp}
		\frac{\partial u}{\partial t}(z,t)-\Delta_z u(z,t) &=0, &&\text{for }(z,t)\in\Omega\times(0,+\infty), \\
		\frac{\partial u}{\partial\nnu}(z,t)&=\Phi(u(z,t)), &&\text{for }(z,t)\in\partial\Omega\times (0,+\infty),
	\end{bvp}
\]
where $\nnu$ denotes the normal vector of $\partial\Omega$ and $\Phi$ should reflect the fact that, if $u(z,t)$ is below the threshold $h(z)$ then the temperature must be raised, while if $u(z,t)$ is greater than $h(z)$, it must be lowered. Moreover, the adjustment should be proportionally related to the difference $\abs{u(z,t)-h(z)}$. Therefore a possible choice is
\begin{align*}
	\Phi(u(z,t))&=\lambda_- ((u(z,t)-h(z))^-)^{p-1}-\lambda_+ ((u(z,t)-h(z))^+)^{p-1} \\
	&=\begin{cases}
		\lambda_- (h(z)-u(z,t))^{p-1}, &\text{if }u(z,t)<h(z),\\
		0, &\text{if }u(z,t)=h(z), \\
		-\lambda_+ (u(z,t)-h(z))^{p-1}, &\text{if }u(z,t)>h(z).
	\end{cases}
\end{align*}
Formally, if $u(z,t)$ is converging to some asymptotic configuration $u(z)$, as $t\to +\infty$, then it must satisfy
\[
\begin{bvp}
	-\Delta u &=0, &&\text{in }\Omega, \\
	\frac{\partial u}{\partial\nnu}&=\Phi(u), &&\text{on }\partial\Omega.
\end{bvp}
\]
This corresponds to \eqref{eq:strong_form_1} in the case $a=0$ and when a portion of $\partial\Omega$ lies on the thin space $\R^N$. In view of this physical interpretation, problem \eqref{eq:strong_form_1} can be said to possess a \emph{stable} nature. We refer to \cite[Section 2]{Danielli2021} and references therein for more details concerning physical motivations of the problem under investigation (see also \cite{Bucur2016} for the fractional framework).

Free boundary problems associated to the fractional Laplace operator have been the object of intensive investigation in  recent years. We refer the interested reader to \cite{Danielli2018}, \cite{PSU2012} for thorough monographs on the subject. For instance, in \cite{Allen2012a}, \cite{Allen2012b} the authors studied the so-called lower dimensional two-phase membrane problem, while in \cite{Allen2015}, \cite{Allen2019} and \cite{Danielli2021} the fractional version of a penalized obstacle problem is considered. A fairly complete picture for the regularity of the free boundary in the ``classical'' lower dimensional obstacle problem is given in \cite{AC2004}, \cite{ACS2008}, \cite{Caffarelli2008}, together with \cite{Garofalo2009}, \cite{GRO2019}.  In \cite{Sire2020} the authors investigated geometric theoretical features of the nodal set of $s$-harmonic functions. We finally quote \cite{Soave2018}, \cite{Soave2019}, where the authors developed new techniques in order to study regularity properties of zero level sets of solutions to local equations with sublinear (and even singular) powers; \cite{Tortone2020} extends some of the results to the nonlocal framework.

 In order to state our main results, we first introduce the appropriate functional setting. For any positive half ball $B_r^+$, we define the weighted Sobolev space $H^{1,a}(B_r^+)=H^1(B_r^+;y^a\dxdy)$ as the completion of $C^\infty(\overline{B_r^+})$ with respect to the norm
\[
	\norm{u}_{H^{1,a}(B_r^+)}:=\left(\int_{B_r^+}y^a (\abs{\nabla u}^2+u^2)\dxdy \right)^{1/2}.
\]
Moreover, we denote by $H^{1,a}_{0,S_r^+}(B_r^+)$ the closure of $C_c^\infty(\overline{B_r^+}\setminus {\overline{S_r^+}})$ in $H^{1,a}(B_r^+)$, where $S_r^+:=\partial B_r^+{\cap\R^{N+1}_+}$.
%\blue{A function $u\in H^{1,a}(\Omega)$, with $\Omega\sub\R^{N+1}$ being an open set, is said to be \emph{$a$-harmonic} in $\Omega$ if
%\begin{equation}\label{eq:a-harmonic}
%	-\dive(y^a\nabla u)=0,\quad\text{in }\Omega
%\end{equation}
%holds in a distributional sense, that is
%\begin{equation}\label{eq:a-harmonic-distr}
%	\int_\Omega y^a\nabla u\cdot\nabla \phi\dxdy=0\quad\text{for all }\phi\in C_c^\infty(\Omega).
%\end{equation}}
 For a fixed function $g\in H^{1,a}(B_1^+)\cap C(\overline{B_1^+})$, we consider the problem
\begin{equation}\label{eq:strong_form_2}
\begin{bvp}
-\dive(y^a\nabla u)&=0, &&\text{in }B_1^+, \\
u&=g, &&\text{on }S_1^+, \\
-\lim_{y\to 0^+}y^a\frac{\partial u}{\partial y}&=\lambda_-(u^-)^{p-1}-\lambda_+(u^+)^{p-1}, &&\text{on }B_1'.
\end{bvp}
\end{equation}
In particular, we say that $u\in H^{1,a}(B_1^+)\cap L^p(B_1')$ is a \emph{solution} of the problem above if $u-g\in H^{1,a}_{0,S_1^+}(B_1^+)$ and
\begin{equation*}
\int_{B_1^+}y^a\nabla u\cdot\nabla\phi\dxdy=\int_{B_1'}(\lambda_-(u^-)^{p-1}-\lambda_+(u^+)^{p-1})\phi\dx,
\end{equation*}
for all $\phi\in H^{1,a}_{0,S_1^+}(B_1^+)$.  The Dirichlet boundary datum on $S_1^+$ in \eqref{eq:strong_form_2} appears (in contrast to \eqref{eq:strong_form_1})  in order to establish existence and uniqueness of a weak solution, see Proposition \ref{prop:exist_unique}. Also, we explicitly note that, for the sake of simplicity, we have taken the thin obstacle $h\equiv 0$.

We carry out our analysis  in the following steps. First of all, we establish the optimal regularity of the solution $u$ in H\"older spaces, see Lemma \ref{lemma:regularity}.  We point out that there is a substantial difference in this respect between the cases $a=0$ and $a\neq 0$. While in the former the optimal H\"older space depends on the exponent $p$, this is no longer true in the latter case.  This is a clear indication of the strong influence of the weight $y^a$, which can be degenerate or singular when $y\to 0^+$, and  affects the behavior of the solution near the thin space $\partial\R^{N+1}_+$. Nevertheless, this weight belongs to the second Muckenhoupt class $A_2$ and therefore it enjoys nice properties, including e.g. Sobolev and trace embeddings, see for instance \cite{Fabes1982}, \cite{Murthy1968} and also Lemma \ref{lemma:trace_ineq}. At this point, we investigate the local behavior of $u(x,y)$ when $y\to 0^+$.
In order to do so, we establish  Almgren-  and  Monneau-type monotonicity formulas that, together with a blow up analysis for a proper rescaling of $u$, provide the asymptotic rate and shape of the solution near a free boundary point, in terms of $a$-harmonic polynomials (see Definition \ref{def:subharm_harm}), which are even in the variable $y$ and homogeneous of some degree. More precisely, for any integer $k\geq 0$, let $\mathds{P}_k^a$ denote the space of polynomials $p\colon\R^{N+1}\to\R$ such that
\begin{gather}
	p(x,-y)=p(x,y)\quad\text{for all }(x,y)\in\R^{N+1},\label{eq:polynom1}\\
	p(\mu z)=\mu^kp(z)\quad\text{for all }z\in\R^{N+1}~\text{and all }\mu\geq 0,\label{eq:polynom2}\\
	-\dive(\abs{y}^a\nabla p)=0\quad\text{in }\R^{N+1}.\label{eq:polynom3}
\end{gather}
One can immediately notice that, given the other conditions, the first one is equivalent to
\[
	\lim_{y\to 0^+}\abs{y}^a\frac{\partial p}{\partial y}=0\quad\text{on }\R^N\times\{0\}.
\]
Also, thanks to \cite[Theorem 1.1]{Sire2019}, if the conditions above are satisfied, no other values of $k$ apart from integer ones are allowed.

We are now able to state our first main result.
\begin{theorem}\label{thm:blow_up_homo}
	Let $u\in H^{1,a}(B_1^+)$ be the unique solution to \eqref{eq:strong_form_2} and let $x_0\in B_1'$. If
\[
		u_{r}^{x_0}(z):=u(rz+x_0),
	\]
then there exists an integer $k\geq 0$ and $p_k^{x_0}\in\mathds{P}_k^a$ such that
	\begin{gather*}
		r^{-k}u^{x_0}_{r}\to p_k^{x_0}\quad \text{in }H^{1,a}(B_1^+)~\text{and }C^{0,\alpha}(\overline{B_1^+}), \\
		r^{-k}u^{x_0}_{r}\to  p_k^{x_0} \quad\text{in }C^{1,\alpha}(B_1'), %\\
		%r^{-k+1}y^a\frac{\partial u_{x_0,r}}{\partial y}\to y^a\frac{\partial p_k^{x_0}}{\partial y}\quad\text{in }C^{0,\alpha}(\overline{B_1^+})
	\end{gather*}
	as $r\to 0^+$, for some $\alpha\in(0,1)$.
	
\end{theorem}
Since the polynomial $p_k^{x_0}$ (which we refer to as \emph{blow-up limit} of $u$ at $x_0$) cannot vanish everywhere on the thin space $\R^N\times\{0\}$ (see Remark \ref{rmk:homo_blow_up}), the previous theorem readily implies the boundary strong unique continuation principle.
\begin{corollary}[Strong unique continuation]
	Let $u\in H^{1,a}(B_1^+)$ be the unique solution of \eqref{eq:strong_form_2} and assume that $u(z)=o(\abs{z}^n)$ as $\abs{z}\to 0$, $z\in \overline{B_1^+}$ for all $n\in \N$. Then $u\equiv 0$ on $B_1^+$.
\end{corollary}
This result tells us that the nodal set of $u$ on the thin space
\begin{equation*}
	\mathcal{Z}(u):={\partial}\{x\in B_1'\colon u(x,0)=0 \}
\end{equation*}
has empty interior in the $\R^N$ topology. Therefore, thanks to the continuity of $u$, it coincides with the free boundary that separates the two phases of the solution
\begin{align*}
	\mathcal{Z}(u)&=\partial\{ u(\cdot,0)\not\equiv 0\} \\
	&=\partial\{ u(\cdot,0)>0 \}\cup\partial\{u(\cdot,0)<0  \}.
\end{align*}
The last part of the present work is devoted to the analysis of the regularity and structural properties of the nodal set $\mathcal{Z}(u)$. We first remark that, when investigating the properties of the free boundary in low dimensions ($N=1,2$) some extra care is needed. Therefore, in order to avoid technicalities and to preserve the main idea, we hereafter assume $N\geq 2$ or $N\geq 3$ (we specify the choice in each statement). A fundamental tool is the notion of frequency. More precisely, we introduce, for any $v\in H^{1,a}(B_1^+)$, $x_0\in B_1'$ and $r>0$
\begin{equation*}
	\mathcal{N}_{x_0}(v,r):=\frac{\displaystyle r\int_{B_r^+(x_0)}y^a\abs{\nabla v}^2\dxdy }{\displaystyle \int_{S_r^+(x_0)}y^a v^2\ds},
\end{equation*}
where
\[
	B_r^+(x_0):=(x_0,0)+B_r^+\quad\text{and}\quad S_r^+(x_0):=(x_0,0)+S_r^+.
\]
We call $\mathcal{N}_{x_0}(v,r)$, as a function of $r$, the \emph{frequency function} of $v$ at $x_0$. In Proposition \ref{prop:lim_N} we prove, if $u$ the solution of \eqref{eq:strong_form_2}, the existence of the limit
\[
	\mathcal{N}_{x_0}(u,0^+):=\lim_{r\to 0^+}\mathcal{N}_{x_0}(u,r).
\]
In addition, we show in Theorem \ref{thm:blow_up} that $\mathcal{N}_{x_0}(u,0^+)$   is a nonnegative integer, and it is the degree of the blow-up limit of $u$ at $x_0$. We refer to $\mathcal{N}_{x_0}(u,0^+)$ as the \emph{frequency} of $u$ at $x_0$. To proceed, we split the nodal set $\mathcal{Z}(u)$ into two disjoint parts
\begin{align*}
	\mathcal{R}(u):&=\mathcal{Z}_1(u), \\
	\mathcal{S}(u):&=\bigcup_{k\geq 2}\mathcal{Z}_k(u),
\end{align*}
where, for any integer $k\geq 1$,
\begin{equation*}
	\mathcal{Z}_k(u):=\{x_0\in \mathcal{Z}(u)\colon \mathcal{N}_{x_0}(u,0^+)=k \}.
\end{equation*}
We call $\mathcal{R}(u)$ and $\mathcal{S}(u)$ the \emph{regular} and \emph{singular} part of the free boundary, respectively. At this point of our analysis, we are confronted again with the strong influence of the weight $y^a$ on the behavior of the solution $u$ near the thin space. Indeed, we perform another classification of the zero points of $u$, based on whether the blow-up limit of the solution at $x_0\in\mathcal{Z}(u)$ depends on the variable $y$ or not. Namely, let us consider the following two spaces of polynomials
\begin{equation*}
	\mathds{P}_k^*:=\{p\in\mathds{P}_k^a\colon p(x,y)=p(x)~\text{for all }(x,y)\in\R^{N+1} \}\quad\text{and}\quad\mathds{P}_k^y:=\mathds{P}_k\setminus\mathds{P}_k^*.
\end{equation*}
It is readily seen that the former can be characterized as
\begin{equation*}
	\mathds{P}_k^*=\{ p\in\mathds{P}_k^a\colon \Delta_x p=0~\text{in }\R^{N+1} \}.
\end{equation*}
Moreover we observe that $\mathds{P}_1^y=\emptyset$, and this leads to our first result about the free boundary, concerning its regular part.
\begin{proposition}\label{prop:reg_part}
	Let $N\geq 2$. The regular part of the free boundary $\mathcal{R}(u)$ is a $C^{1,\alpha}$-hypersurface of $B_1'$ and
	\[
	\mathcal{R}(u)=\{x_0\in \mathcal{Z}(u)\colon \abs{\nabla_x u(x_0,0)}\neq 0 \}.
	\]
\end{proposition}
On the other hand, since $\mathds{P}_k^y\neq \emptyset$ for $k\geq 2$, we cannot expect the same regularity for the singular set. However, we are able to prove a stratification result which describes the structure of $\mathcal{S}(u)$. Once more, we must distinguish between points for which the blow-up limit depends on the variable $y$, and points for which this does not happen. If we define, for $p_k^{x_0}$  as in Theorem \ref{thm:blow_up_homo},
\begin{align*}
\mathcal{Z}_k^*(u):&=\{x_0\in\mathcal{Z}_k(u)\colon p_k^{x_0}\in\mathds{P}_k^*  \}, \\
\mathcal{Z}_k^y(u):&=\{x_0\in\mathcal{Z}_k(u)\colon p_k^{x_0}\in\mathds{P}_k^y  \}=\mathcal{Z}_k(u)\setminus\mathcal{Z}_k^*(u),
\end{align*}
then  the singular set naturally splits into the following two disjoint parts
\[
\mathcal{S}^y(u):=\bigcup_{k\geq 2}\mathcal{Z}_k^y(u)\quad\text{and}\quad\mathcal{S}^*(u):=\bigcup_{k\geq 2}\mathcal{Z}_k^*(u).
\]
In order to  prove their stratified structure, we need the notion of dimension of the singular set at one of its points.
\begin{definition}
	For any $x_0\in\mathcal{Z}_k(u)$, we define the \emph{dimension} of $\mathcal{Z}_k(u)$ at $x_0$ as
	\[
	d^{x_0}_k:=\dim \left\{ \xi\in\R^N\colon \nabla_x p_k^{x_0}(x,0)\cdot \xi=0~\text{for all }x\in\R^N  \right\}.
	\]
\end{definition}
 We observe that, since $p_k^{x_0}\not\equiv 0$ on $\R^N\times\{0\}$, we have that $0\leq d_k^{x_0}\leq N-1$. Finally, let us define
\[
\mathcal{Z}_k^n(u):=\left\{ x_0\in\mathcal{Z}_k(u)\colon d_k^{x_0}=n  \right\}.
\]

We are now able to describe the structure of the singular set. Roughly speaking, we show that $\mathcal{Z}_k(u)$ is contained in a $d_k^{x_0}$-dimensional manifold near $x_0$.

\begin{theorem}\label{thm:stratification}
	Let $N\geq 3$. The set $\mathcal{S}^y(u)$ is contained in a countable union of $(N-1)$-dimensional $C^1$ manifolds, while $\mathcal{S}^*(u)$ is contained in a countable union of $(N-2)$-dimensional $C^1$ manifolds. Furthermore
	\[
	\mathcal{S}^y(u)=\bigcup_{n=0}^{N-1}\mathcal{S}_n^y(u)\quad\text{and}\quad\mathcal{S}^*(u)=\bigcup_{n=0}^{N-2}\mathcal{S}_n^*(u),
	\]
	where
	\[
	\mathcal{S}_n^y(u):=\bigcup_{k\geq 2}\mathcal{Z}_k^y(u)\cap\mathcal{Z}_k^n(u)\quad\text{and}\quad\mathcal{S}_n^*(u):=\bigcup_{k\geq 2}\mathcal{Z}_k^*(u)\cap\mathcal{Z}_k^n(u)
	\]
	and they are contained in a countable union of $n$-dimensional $C^1$ manifolds.
\end{theorem}

The last result of our paper yields estimates on the Hausdorff dimension of the nodal set and of its regular and singular part.

\begin{theorem}\label{thm:hausd}
	Let $N\geq 2$. Let $u\in H^{1,a}(B_1^+)$ be the unique solution of \eqref{eq:strong_form_2} and assume $u\not\equiv 0$ on $B_1'$. Then
	\begin{equation}\label{eq:hausd_th_1}
	\text{either}\quad\mathcal{Z}(u)=\emptyset\quad\text{or}\quad\dim_H (\mathcal{Z}(u))=N-1.
	\end{equation}
	Furthermore, in the latter case
	\begin{equation}\label{eq:hausd_th_2}
	\text{either}\quad\mathcal{R}(u)=\emptyset\quad\text{or}\quad \dim_H (\mathcal{R}(u))=N-1
	\end{equation}
	and
	\begin{equation}\label{eq:hausd_th_3}
	\text{either}\quad\mathcal{S}(u)=\emptyset\quad\text{or}\quad\dim_H (\mathcal{S}(u))\leq N-1.
	\end{equation}
\end{theorem}

The proof  relies on Theorem \ref{thm:blow_up_homo} and Federer's Reduction Principle. We observe that the previous result immediately implies the boundary unique continuation from sets of positive $N$-dimensional measure.

\begin{corollary}
	Let $N\geq 2$. Let $u\in H^{1,a}(B_1^+)$ be the unique solution to \eqref{eq:strong_form_2}. If $\abs{\mathcal{Z}(u)}_N>0$, then $u\equiv 0$ in $B_1^+$.
\end{corollary}

	We conclude the introduction with a couple of comments. We first emphasize that the assumption $p\geq 2$ is made in view of the fact that the value $2$ delineates a threshold, making the problem superlinear: indeed, in the sublinear regime $p\in (1,2)$, even though some of the results of the present paper still hold true, the structure of the of the nodal set might be qualitatively different and this is somehow suggested by the available results in the local case $a=-1$ (i.e. $s=1$). For instance, already in dimension $1$, the ODE $-u''=\lambda \abs{u}^{p-2}u$ (corresponding to the case $\lambda_-=\lambda_+=\lambda>0$) admits nontrivial solutions with arbitrary large nodal set. Therefore, the equation
	\[
		-\Delta u=\lambda_-(u^-)^{p-1}-\lambda_+(u^+)^{p-1}
	\]
	does not enjoy the unique continuation property when $p\in (1,2)$ (see e.g. \cite{Ruland2018}). We refer to \cite{Fotouhi2017} for the study of the free boundary $\partial\{u<0\}\cup\partial\{u>0\}$ of the problem above. In the nonlocal framework a similar situation is expected, even if we cannot exclude some peculiar features to emerge: anyway this requires further investigation, which may be interesting to pursue in the future.
	
	We finally mention some further possible research topics, which are left open by our present work. In the same spirit as other classical problems, one might investigate higher regularity of the free boundary, concerning both the regular and the singular part. In addition to the classical works, starting points in this direction can be some recent outstanding approaches, such as the ones developed in \cite{MR3491531,Colombo2020,Figalli2020} (see also \cite{DSV2021,Figalli2019,Spolaor2021}). Another interesting point one may investigate concerns the possible vanishing orders of solutions to \eqref{eq:strong_form_2}. We proved in Theorem \ref{thm:blow_up_homo} that the vanishing order must be a nonnegative integer, but still it is not clear whether it could be any possible element of $\mathbb{N}$ or not. For instance there could be an upper bound (such as in the linear case, see \cite{Ruland2017}) or some integer numbers might be excluded.

\subsection{Outline of the paper}

The paper is organized as follows. In Section \ref{sec:regularity} we provide some basic properties of the solution,  including the optimal regularity in H\"older spaces. Section \ref{sec:monotonicity_form} is devoted to the proof of  monotonicity of a perturbed Almgren type frequency function and of a Monneau functional. These results are then applied in Section \ref{sec:blow_up} to perform a blow-up analysis for a suitable scaling of the solution, which leads to the proof of Theorem \ref{thm:blow_up_homo}. Finally, in Section \ref{sec:free_boundary} we prove the structure of the free boundary and Hausdorff dimension estimates.

%\subsection{Notation}

%Here we list some notation used throughout the paper.

%\begin{itemize}
	%\item $\R^{N+1}_+:=\{(x,y)\colon x\in\R^N,~y>0 \}=\R^N\times \R_+$ denotes the positive half space;
	%\item $B_r^+(x_0):=\{z\in\R^{N+1}_+\colon\abs{z-x_0}<r  \}$ and $S_r^+(x_0):=\{z\in\R^{N+1}_+\colon \abs{z-x_0}=r \}$ denote, respectively, the positive half balls and half spheres, for any $x_0\in\R^N$ and $r>0$;
	%\item $B_r'(x_0):=\partial B_r^+(x_0)\setminus \overline{S_r^+(x_0)}$ for balls on the thin space;
	%\item in case of balls and spheres centered at the origin, we drop the center in the notation, i.e. $B_r^+:=B_r^+(0)$, $S_r^+:=S_r^+(x_0)$ and $B_r':=B_r'(x_0)$.
%\end{itemize}

\section{Optimal regularity and main properties of the solution}\label{sec:regularity}

In this section we higlight some important features of the solution and we establish its optimal regularity.

In order to prove existence and uniqueness of a solution to problem \eqref{eq:strong_form_2}, it is useful to set up a variational framework. Hence, we introduce, for any $v\in H^{1,a}(B_1^+)\cap L^p(B_1')$, the following energy functional
\begin{equation}\label{eq:functional}
	J(v):=\frac{1}{2}\int_{B_1^+}y^a\abs{\nabla u}^2\dxdy+\frac{1}{p}\int_{B_1'}(\lambda_-(u^-)^p+\lambda_+(u^+)^p)\dx
\end{equation}
and the constraint
\begin{equation}\label{eq:constraint}
	\Theta=\Theta_g:=\left\{v\in H^{1,a}(B_1^+){\cap L^p(B_1')}\colon v-g\in H^{1,a}_{0,S_1^+}(B_1^+) \right\}.
\end{equation}
We will need the following Poincar\'e inequality.
\begin{lemma}\label{lemma:poincare}
	For all $r>0$ and for all $v\in H^{1,a}(B_r^+)$, there holds
	\[
	\frac{N+a}{r^2}\int_{B_r^+}y^a v^2\dxdy\leq \int_{B_r^+}y^a\abs{\nabla v}^2\dxdy+\frac{1}{r} \int_{S_r^+}y^av^2\ds .
	\]
\end{lemma}
\begin{proof}
	Just by integration of the identity
	\[
	\dive(y^a u^2 z)=2y^a u\nabla u\cdot z+(N+a+1)y^a u^2
	\]
	over $B_r^+$ and Young's inequality (with $z=(x,y)$).
\end{proof}

This variational setting allows us to prove the following.

\begin{proposition}\label{prop:exist_unique}
	There exists a unique solution to problem \eqref{eq:strong_form_2}, i.e. a function $u\in \Theta$ such that
	\begin{equation*}
		\int_{B_1^+}y^a\nabla u\cdot\nabla\phi\dxdy=\int_{B_1'}(\lambda_-(u^-)^{p-1}-\lambda_+(u^+)^{p-1})\phi\dx,
	\end{equation*}
	for all $\phi\in H^{1,a}_{0,S_1^+}(B_1^+)$. In particular, the solution $u$ is the unique minimum of the functional $J$ under the constraint $\Theta$. Here $\Theta$ and $J$ are  as in \eqref{eq:functional} and \eqref{eq:constraint}.
\end{proposition}
\begin{proof}
	Thanks to the form of the functional $J$, the proof of existence is standard and is based on Lemma \ref{lemma:poincare}, weak lower-semicontinuity of the $H^{1,a}$-norm, compactness of the trace embedding $H^{1,a}(B_1^+)\hookrightarrow L^2(B_1')$, and Fatou's Lemma. Uniqueness follows from convexity of the functions $t\mapsto (t^{\pm})^p$.
\end{proof}

The first notable property of the solution is that its positive and negative part are subharmonic functions, with respect to the operator $-\dive(y^a\nabla(\cdot))$, in $B_1^+\cup B_1'$. We first give the precise definition.
\begin{definition}\label{def:subharm_harm}
	We say that a function $v\in H^{1,a}(B_1^+)$ is \emph{$a$-subharmonic} in $B_1^+\cup B_1'$ if
	\[
	\int_{B_1^+}y^a\nabla v\cdot\nabla \phi\dxdy\leq 0,
	\]
	for all $\phi\in H^{1,a}_{0,S_1^+}(B_1^+)$ such that $\phi\geq 0$ a.e. Moreover, we say that $v\in H^{1,a}(B_1^+)$ is \emph{$a$-harmonic} in $B_1^+\cup B_1'$ if
	\[
	\int_{B_1^+}y^a\nabla v\cdot\nabla \phi\dxdy=0,
	\]
	for all $\phi\in H^{1,a}_{0,S_1^+}(B_1^+)$.		Finally, we say that $v\in H^{1,a}_{\textup{loc}}(\Omega)$ is \emph{$a$-harmonic} in an open $\Omega\sub\R^{N+1}$ if
	\[
	\int_{\Omega}y^a\nabla v\cdot\nabla \phi\dxdy=0,
	\]
	for all $\phi\in C_c^\infty(\Omega)$.
\end{definition}

We now prove $a$-subharmonicity of the positive and negative part of solutions to \eqref{eq:strong_form_2}.

\begin{lemma}\label{lemma:subharm}
	Let $u\in H^{1,a}(B_1^+)$ be the unique solution to \eqref{eq:strong_form_2}. Then $u^+$ and $u^-$ are $a$-subharmonic in $B_1^+\cup B_1'$.% i.e.
	%\[
	%	\int_{B_1^+}y^a\nabla u^{\pm}\cdot\nabla \phi\dxdy\leq 0,
	%\]
	%for all $\phi\in H^{1,a}_{0,S_1^+}(B_1^+)$ such that $\phi\geq 0$ a.e.
\end{lemma}
\begin{proof}
	Let $\phi\in H^{1,a}_{0,S_1^+}(B_1^+)$ such that $\phi\geq 0$ a.e. and, for $\epsilon>0$ small, let $u_\epsilon:=(u-\epsilon\phi)^+-u^-$. One can immediately notice that
	\begin{equation}\label{eq:subharmm_1}
		u_\epsilon^+=(u-\epsilon\phi)^+\leq u^+\quad\text{and that}\quad u_\epsilon^-=u^-.
	\end{equation}
	Since $u_\epsilon=u$ on $S_1^+$, by minimality and \eqref{eq:subharmm_1} we have
	\[
		\int_{B_1^+}y^a\abs{\nabla u^+}^2\dxdy\leq \int_{B_1^+}y^a\abs{\nabla (u-\epsilon\phi)^+}^2\dxdy.
	\]
	This yields
	\begin{multline*}
		\int_{B_1^+}y^a\abs{\nabla u}^2\chi_{\{u\geq 0\}}\dxdy\leq \int_{B_1^+}y^a\abs{\nabla (u-\epsilon\phi)}^2\chi_{\{u\geq \epsilon\phi\}}\dxdy \\ \leq \int_{B_1^+}y^a\abs{\nabla (u-\epsilon\phi)}^2\chi_{\{u\geq 0\}}\dxdy.
	\end{multline*}
	By expanding the last term we obtain
	\[
		\int_{B_1^+}y^a\nabla u^+\cdot\nabla \phi\dxdy\leq \frac{\epsilon}{2}\int_{B_1^+}y^a	\abs{\nabla\phi}^2\chi_{\{u\geq 0\}}\dxdy	
	\]
	for all $\epsilon>0$, which implies $a$-subharmonicity of $u^+$. Similarly, by letting $u_\epsilon:=u^+-(u+\epsilon\phi)^-$ we can prove $a$-subharmonicity of $u^-$.
\end{proof}

Since $u^+$ and $u^-$ are $a$-subharmonic, the weak maximum principle (together with boundedness of $g$ on $S_1^+$ and a reflection argument) immediately implies the following, see \cite[Theorem 2.2.2]{Fabes1982} and \cite[Theorem 6.7]{Murthy1968}.

\begin{lemma}[Weak maximum principle]\label{lemma:weak_max_princ}
	Let $u\in H^{1,a}(B_1^+)$ be the unique solution of problem \eqref{eq:strong_form_2}. Then
	\[
		\min\{0,\inf_{S_1^+}g\}\leq u\leq \max\{0,\sup_{S_1^+}g\}
	\]
	a.e. in $B_1^+$. In particular $u\in L^\infty(B_1^+)$ and
	\[
	\norm{u}_{L^\infty(B_1^+)}\leq \norm{g}_{L^\infty(S_1^+)}.
	\]
\end{lemma}

We also recall the validity of mean value inequalities for $a$-subharmonic functions. The following result immediately follows from \cite[Lemma A.1]{Wang2016} and an even-in-$y$ reflection.

\begin{lemma}\label{lemma:mean_value_ineq}
	Let $v\in H^{1,a}\cap C(B_1^+\cup B_1')$ be $a$-subharmonic in $B_1^+\cup B_1'$, that is
	\[
	\int_{B_1^+}y^a\nabla v\cdot\nabla \phi\dxdy\leq 0,
	\]
	for all $\phi\in H^{1,a}_{0,S_1^+}(B_1^+)$ such that $\phi\geq 0$ a.e. Then
	\[
		v(x_0,0)\leq \frac{\alpha_{N,a}}{r^{N+a+1}}\int_{B_r^+(x_0)}y^a v\dxdy\quad\text{and}\quad	v(x_0,0)\leq \frac{\beta_{N,a}}{r^{N+a}}\int_{S_r^+(x_0)}y^a v\ds
	\]
	for all $x_0\in B_1'$ and $r>0$ such that $B_r^+(x_0)\subset\subset B_1^+$, where
	\[
		\alpha_{N,a}:=\int_{B_1^+}y^a\dxdy\quad\text{and}\quad\beta_{N,a}:=\int_{S_1^+}y^a\ds.
	\]
\end{lemma}

Boundedness of the solution, stated in Lemma \ref{lemma:weak_max_princ}, allows us to deduce its regularity. In particular, from \cite[Theorem 6.4]{Allen2015} (see also \cite[Theorem 1.5 and 1.6]{Sire2019}) and \cite[Theorem 1.1]{Danielli2021} we deduce the following.

\begin{lemma}\label{lemma:regularity}
	Let $u\in H^{1,a}(B_1^+)$ be the unique solution to problem \eqref{eq:strong_form_2}. There holds
	\begin{itemize}
		\item if $a>0$ then $u\in C^{0,1-a}(B_{1/4}^+)$ and
		\[
		\norm{u}_{C^{0,1-a}(B_{1/4}^+)}\leq C(N,a,p,\norm{u}_{L^\infty(B_1^+)});
		\]
		\item if $a< 0$ then $u\in C^{1,-a}(B_{1/4}^+\cup B_{1/4}')$ and
		\[
		\norm{u}_{C^{1,-a}(B_{1/4}^+\cup B_{1/4}')}\leq C(N,a,p,\norm{u}_{L^\infty(B_1^+)});
		\]
		\item if $a=0$ then $u\in C^{1,\alpha}(B_{1/4}^+\cup B_{1/4}')$ for some $\alpha\in (0,1)$ (depending on $p$) and
		\[
		\norm{u}_{C^{1,\alpha}(B_{1/4}^+\cup B_{1/4}')}\leq C(N,\norm{u}_{L^\infty(B_1^+)}).
		\]
	\end{itemize}
\end{lemma}

\begin{remark}
	We point out that the result is optimal when $a\neq 0$. Indeed, let $\Psi(x)$ be harmonic in $\R^N$. Then, the function $u(x,y)=(y^{1-a}+1-a)\Psi(x)$ is a solution of \eqref{eq:strong_form_2} with $\lambda_+=\lambda_-=1$ and $p=2$, and so Lemma \ref{lemma:regularity} is sharp. On the contrary, when $a=0$ the regularity is actually higher, depending on the exponent $p$ (we refer to \cite{Danielli2021} for a detailed exposition).
\end{remark}

Even though the global regularity cannot exceed the one just stated, if we look separately at the derivatives in $x$ and the weighted derivative in $y$, we are able to prove something more.

\begin{lemma}[{\cite[Proposition 3.1]{Tan2011}}]\label{lemma:tan2011}
	Let $U\in H^{1,a}(B_1^+)$ be a weak solution of
	\[
		\begin{bvp}
			-\dive(y^a\nabla U)&=0, &&\text{in }B_1^+, \\
			-\lim_{y\to 0^+}y^a\frac{\partial U}{\partial y}&=b\, U+d, &&\text{on }B_1',
		\end{bvp}
	\]
	with $b,d\in L^\infty(B_1')$. Then $U\in L^\infty(B_{1/2}^+)$ and
	\[
		\norm{U}_{L^\infty(B_{1/2}^+)}\leq C(\norm{U}_{L^{2,a}(B_1^+)}+\norm{d}_{L^\infty(B_1')}),
	\]
	for a certain $C>0$ depending on $N,a$ and $\norm{b}_{L^\infty(B_1')}$.
\end{lemma}

\begin{remark}\label{rmk:tan2011}
	By inspecting the proof of the previous lemma, which is based on an iterative Moser-type argument, one can find a more explicit expression for the constant $C$ and emphasize the dependence on $\norm{b}_{L^\infty(B_1')}$. More precisely
	\[
		C=c_1(\max\{c_2,\norm{b}_{L^\infty(B_1')}+1\})^{c_3},
	\]
	where the $c_i$'s are positive constants depending on $N$ and $a$.
\end{remark}

\begin{lemma}\label{lemma:high_reg}
	Let $u\in H^{1,a}(B_1^+)$ be the unique solution to problem \eqref{eq:strong_form_2}. Then there exists $\alpha\in(0,1)\cap(0,1-a)$, $R_0\in (0,1/4)$ and $C>0$, possibly depending on $N,a,p,\norm{g}_{L^\infty(S_1^+)}$, such that
	\begin{gather}
		\norm{\nabla_{x}u}_{C^{0,\alpha}(B_{R_0}^+)}\leq C\norm{\nabla u}_{L^{2,a}(B_1^+)}, \label{eq:high_reg_th_1}\\
		\norm{y^a\frac{\partial u}{\partial y}}_{C^{0,\alpha}(B_{R_0}^+)}\leq C. \label{eq:high_reg_th_2}
	\end{gather}
\end{lemma}
\begin{proof}
	For $i=1,\dots,N$, let $\bm{e}_i$ be the $i$-th vector of the standard basis of $\R^{N+1}$ and, for $h\in\R$ sufficiently small and $z=(x,y)\in B_{1/4}^+$, let
	\[
		u^h(x,y)=u_i^h(x,y):=\frac{u(x+h\bm{e}_i,y)-u(x,y)}{\abs{h}}.
	\]
	This function weakly satisfies
	\begin{equation*}
	\begin{bvp}
	-\dive(y^a\nabla u^h)&=0, &&\text{in }B_{1/4}^+, \\
	-\lim_{y\to 0^+}y^a\frac{\partial u^h}{\partial y}&=V_h u^h, &&\text{on }B_{1/4}',
	\end{bvp}
	\end{equation*}
	where
	\begin{multline*}
	V_h(x)=\bigg[\lambda_-\,\frac{(u^-(x+h\bm{e}_i,0))^{p-1}-(u^-(x,0))^{p-1}}{u(x+h\bm{e}_i,0)-u(x,0)}\\-\lambda_+\,\frac{(u^+(x+h\bm{e}_i,0))^{p-1}-(u^+(x,0))^{p-1}}{u(x+h\bm{e}_i,0)-u(x,0)}\bigg]\chi_{\{u^h(x)\neq 0\}}(x).
	\end{multline*}
	From Lemma \ref{lemma:tan2011} we know that $u^h\in L^\infty(B_{r_1}^+)$, for some $r_1\in(0,1/4)$ and that
	\begin{equation}\label{eq:high_reg_1}
	\norm{u_h}_{L^\infty(B_{r_1}^+)}\leq C\norm{u^h}_{L^{2,a}(B_{1/4}^+)},
	\end{equation}
	with $C_1>0$ depending on $N,a$ and $\norm{V^h}_{L^\infty(B_{1/4}')}$. Thanks to the Lipschitz continuity of the function $t\mapsto (t^\pm)^{p-1}$ on real compact intervals and to Lemma \ref{lemma:weak_max_princ}, we have that $V_h\in L^\infty(B_{1/4}')$ and its $L^\infty$-norm is bounded by a constant that depends on $p$, $\lambda_-$, $\lambda_+$ and $\norm{u}_{L^\infty(B_1^+)}$, uniformly with respect to $h$. More precisely, a direct study of the function
	\[
		(t_1,t_2)\mapsto \frac{\abs{(t_1^\pm)^{p-1}-(t_2^\pm)^{p-1}}}{\abs{t_1-t_2}}
	\]
	tells us that
	\begin{equation}\label{eq:high_reg_2}
		\norm{V_h}_{L^\infty(B_{1/4}')}\leq (\lambda_-+\lambda_+)(p-1)\norm{u}_{L^\infty(B_1^+)}^{p-2}\quad\text{uniformly in }h.
	\end{equation}
	Combining this fact with \eqref{eq:high_reg_1}, in view of Remark \ref{rmk:tan2011} and Lemma \ref{lemma:weak_max_princ}, we have that
	\begin{equation}\label{eq:high_reg_3}
		\norm{u^h}_{L^\infty(B_{r_1}^+)}\leq \tilde{C}_1\norm{u^h}_{L^{2,a}(B_{1/4}^+)},
	\end{equation}
	with $\tilde{C}_1>0$ depending on $N,a,\lambda_+,\lambda_-,p$ and $\norm{g}_{L^\infty(S_1^+)}$. From \cite[Theorem 1.5]{Sire2019} we know that, for any $\alpha\in (0,1)\cap(0,1-a)$ there exists $r_2\in (0,1/4)$ and $C_2>0$, depending on $N,a,\alpha$ such that
	\begin{equation}\label{eq:high_reg_4}
		\norm{u^h}_{C^{0,\alpha}(B_{r_2}^+)}\leq C_2\left[ \norm{u^h}_{L^{2,a}(B_{r_1}^+)}+\norm{V_h \,u^h}_{L^\infty(B_{r_1}^+)} \right].
	\end{equation}
	Actually, \cite[Theorem 1.5]{Sire2019} (see also \cite[Theorem 8.2]{Sire2019}) states this result for $V_h u^h\in L^q$ for some $q$ sufficiently large; nevertheless, by scanning the proof, one can check that it still holds true for $q=\infty$ without any change. On the other hand, by \eqref{eq:high_reg_2} and \eqref{eq:high_reg_3} we have that
	\[
		\norm{V_h \,u^h}_{L^\infty(B_{r_1}^+)}\leq \bar{C}_2\norm{u^h}_{L^{2,a}(B_{1/4}^+)},
	\]
	with $\bar{C}_2>0$ depending on $N,a,\lambda_+,\lambda_-,p$. Therefore, combining this with \eqref{eq:high_reg_4}, we deduce that
	\begin{equation}\label{eq:high_reg_5}
		\norm{u^h}_{C^{0,\alpha}(B_{r_2}^+)}\leq \tilde{C}_3\norm{u^h}_{L^{2,a}(B_{1/4}^+)}
	\end{equation}
	with $\tilde{C}_3>0$ depending on $N,a,\alpha,\lambda_+,\lambda_-,p$. Since $u\in H^{1,a}(B_1^+)$, we have that
	\[
		\norm{u^h}_{L^{2,a}(B_{1/4}^+)}\to \norm{\frac{\partial u}{\partial x_i}}_{L^{2,a}(B_{1/4}^+)}\quad\text{as }\abs{h}\to 0.
	\]
	 Therefore, by the Ascoli-Arzelà Theorem, up to a subsequence,
	 \[
	 	u^h \to  \frac{\partial u}{\partial x_i}\quad\text{uniformly in }\overline{B_{r_2}^+}~\text{as }\abs{h}\to 0
	 \]
	 and then, passing to the limit in \eqref{eq:high_reg_5} yields
	 \[
	 	\norm{\frac{\partial u}{\partial x_i}}_{C^{0,\alpha}(B_{r_2}^+)}\leq \tilde{C}_3\norm{\frac{\partial u}{\partial x_i}}_{L^{2,a}(B_{1/4}^+)}.
	 \]
	By iterating the argument for any $i=1,\dots,N$ we have that \eqref{eq:high_reg_th_1} holds. Estimate \eqref{eq:high_reg_th_2} follows from Lemma \ref{lemma:regularity} and \cite[Lemma 4.5]{Cabre2014}.
\end{proof}

\begin{remark}\label{rmk:smoothness}
	By standard elliptic regularity theory, since the weight $y^a$ is uniformly bounded in every compact subset of $B_1^+$, we have that $u\in C^\infty(B_1^+)$. Moreover, by a covering argument in $B_1^+\cup B_1'$, we can say that $u$, $\nabla_x u$ and $y^a\partial_y u$ are in $C^{0,\min\{1,1-a\}}_{\textup{loc}}(B_1^+\cup B_1')$.
\end{remark}

In view of the regularity of the solution, we have also that a strong maximum principle holds.

\begin{lemma}\label{lemma:str_max_princ}
	Let $u\in H^{1,a}(B_1^+)$ be the unique solution to \eqref{eq:strong_form_2} and assume $g\not\equiv 0$ on $S_1^+$. Then
	\begin{gather*}
		u>0\quad\text{in }B_1^+\cup B_1'\quad\text{if}\quad g\geq 0\quad\text{on }S_1^+, \\
		u<0\quad\text{in }B_1^+\cup B_1'\quad\text{if}\quad g\leq 0\quad\text{on }S_1^+.
	\end{gather*}
\end{lemma}
\begin{proof}
	In view of Remark \ref{rmk:smoothness}, the proof follows from Lemma \ref{lemma:weak_max_princ}, Hopf Lemma and boundary Hopf Lemma for $A_2$-weighted elliptic equations, which can be found respectively in \cite[Corollary 2.3.10]{Fabes1982} and \cite[Proposition 4.11]{Cabre2014}.
\end{proof}

\section{Monotonicity Formulas}\label{sec:monotonicity_form}

In this section we introduce a \emph{perturbed} Almgren type frequency function, suitable for our problem, and we prove its monotonicity. As a consequence we deduce the existence of the limit of the \emph{standard} Almgren frequency function of the solution, as the radius of the underlying ball vanishes. Finally, we introduce a Monneau type functional and we show it is nondecreasing with respect to the radius.

For any $v\in H^{1,a}(B_1^+)\cap L^p(B_1')$, $x_0\in B_1'$ and $r\in (0,R_0)$ (with $R_0$ as in Lemma \ref{lemma:high_reg}), such that $B_r^+(x_0)\sub B_{R_0}^+$, we define the functions
\begin{equation}\label{eq:def_D_H}
	D_{x_0}(v,r):=\frac{1}{r^{N+a-1}}\int_{B_r^+(x_0)}y^a\abs{\nabla v}^2\dxdy,\quad H_{x_0}(v,r):=\frac{1}{r^{N+a}}\int_{S_r^+(x_0)}y^a v^2\ds
\end{equation}
and the classical \emph{frequency} function
\begin{equation*}
	\mathcal{N}_{x_0}(v,r):=\frac{D_{x_0}(v,r)}{H_{x_0}(v,r)}.
\end{equation*}
Moreover we let
\begin{equation}\label{eq:def_F_G}
	F(v):=\lambda_-(v^-)^p+\lambda_+(v^+)^p,\quad G_{x_0}(v,r):=\frac{1}{r^{N+a-1}}\int_{B_r'(x_0)}F(v)\dx.
\end{equation}
Then we introduce the \emph{perturbed frequency} function
\begin{equation*}
	\tilde{\mathcal{N}}_{x_0}(v,r):=\frac{\tilde{D}_{x_0}(v,r)}{H_{x_0}(v,r)},\quad\text{where}\quad \tilde{D}_{x_0}(v,r):=D_{x_0}(v,r)+\frac{2}{p}G_{x_0}(v,r).
\end{equation*}
The first goal of the section is to prove monotonicity of $\tilde{\mathcal{N}}_{x_0}(u,r)$ with respect to $r\in (0,R_0)$, where $u$ is the unique solution to problem \eqref{eq:strong_form_2}. Hereafter, when we compute the functionals above at the origin $x_0=0$ and for $v=u$, we may drop this dependence in our notation and simply write $\tilde{\mathcal{N}}(r)$. Without loss of generality, in this section we can assume $x_0=0$.

\begin{lemma}\label{lemma:R-N}
	Let $v\in C^\infty(\overline{B_1^+})$. There holds the following formula, for $z\in B_1^+$,
	\begin{equation}\label{identity1}
		\dive(y^a\abs{\nabla v}^2 z-2y^a(\nabla v\cdot z)\nabla v)=(N+a-1)y^a\abs{\nabla v}^2-2(\nabla v\cdot z)\dive(y^a\nabla v).
	\end{equation}
\end{lemma}
\begin{proof}
		Since the proof is essentially just computations, we only sketch it. We first have that
		\begin{multline}\label{eq:R-N_01}
			\dive(y^a\abs{\nabla v}^2 z-2y^a(\nabla v\cdot z)\nabla v) \\ =(N+a+1)y^a\abs{\nabla v}^2+y^a z\cdot\nabla(\abs{\nabla v}^2)-2(\nabla v\cdot z)\dive(y^a\nabla v)-2y^a\nabla v\cdot\nabla(\nabla v\cdot z).
		\end{multline}
		Moreover it is easy to check that the following identity holds
		\[
			z\cdot\nabla(\abs{\nabla v}^2)=2\nabla v\cdot\nabla(\nabla v\cdot z)-2\abs{\nabla v}^2.
		\]
		By plugging this into \eqref{eq:R-N_01} we obtain the desired conclusion.
\end{proof}

\begin{lemma}\label{lemma:R-N_int}
	For a.e. $r\in (0,R_0)$ we have that
	\begin{multline}\label{eq:R-N_int_th_1}
		\frac{1}{r^{N+a-1}}\int_{S_r^+}y^a\abs{\nabla u}^2\dxdy=\frac{N+a-1}{r}\tilde{D}(r)+\frac{2}{r^{N+a-1}}\int_{S_r^+}y^a\left(\frac{\partial u}{\partial\nnu}\right)^2\ds\\-\frac{2}{pr^{N+a-1}}\int_{\partial B_r'}F(u)\dsi+\frac{2(1-a)}{p r}G(r).
	\end{multline}
	and
	\begin{equation}\label{eq:R-N_int_th_2}
		\int_{S_r^+}y^a u\,\frac{\partial u}{\partial \nnu}\ds=\int_{B_r^+}y^a\abs{\nabla u}^2\dxdy+\int_{B_r'}(\lambda_-(u^-)^p+\lambda_+(u^+)^p)\dx.
	\end{equation}
\end{lemma}
\begin{proof}
	Since similar computations can be found in literature, we give a sketch of the proof and we refer to, among others, \cite[Theorem 3.7]{Fall2014}  for a precise justification. In order to prove \eqref{eq:R-N_int_th_1} we first integrate \eqref{identity1} over $B_r^+$, which, taking into account the equation \eqref{eq:strong_form_2} satisfied by $u$, gives
	\begin{multline}\label{eq:R-N_1}
		(N+a-1)\int_{B_r^+}y^a\abs{\nabla u}^2\dxdy \\ =r\int_{S_r^+}y^a\abs{\nabla u}^2\ds-2r\int_{S_r^+}y^a\left(\frac{\partial u}{\partial\nnu}\right)^2\ds-2\int_{B_r'}\left(\lambda_-(u^-)^{p-1}-\lambda_+(u^+)^{p-1}\right)(\nabla_x u\cdot x)\dx.
	\end{multline}
	 Now, integrating by parts in $B_r'$ and using the identity
	\[
		(u^{\pm})^{p-1}\nabla_x u^{\pm}\cdot x=\frac{1}{p}\nabla_x(u^{\pm})^p\cdot x
	\]
	we obtain the following expression for the last term in \eqref{eq:R-N_1}
	\[
		\int_{B_r'}\left(\lambda_-(u^-)^{p-1}-\lambda_+(u^+)^{p-1}\right)(\nabla_x u\cdot x)\dx=-\frac{r}{p}\int_{\partial B_r'}F(u)\dsi +\frac{N}{p}\int_{B_r'}F(u)\dx,
	\]
	where $F$ is as in \eqref{eq:def_F_G}. Combining the previous identity with \eqref{eq:R-N_1} and reorganizing the terms yields \eqref{eq:R-N_int_th_1}. Finally, \eqref{eq:R-N_int_th_2} comes from multiplying \eqref{eq:strong_form_2} by $u$ and integrating over $B_r^+$.
\end{proof}

\begin{lemma}\label{lemma:well_pos_H}
	We have that $H(r)>0$ for all $r\in (0,R_0)$.
\end{lemma}
\begin{proof}
	If we assume $H(\bar{r})=0$ for a certain $\bar{r}\in(0,R_0)$, then we reach a contradiction thanks to \eqref{eq:R-N_int_th_2} and unique continuation principle (see e.g. \cite[Theorem 7.1]{Simon1980a}).	
\end{proof}

In the following lemma we state the regularity of the function $H$ and we compute its derivative.

\begin{lemma}\label{lemma:H'}
	We have that $H\in C^1(0,R_0)$ and
	\begin{align}
		H'(r)&=\frac{2}{r^{N+a}}\int_{S_r^+}y^a u\,\frac{\partial u}{\partial \nnu}\ds \label{eq:H'_th_1}\\
		&=\frac{2}{r}\left[\tilde{D}(r)+\left(1-\frac{2}{p}\right)G(r)\right]\geq 0, \label{eq:H'_th_2}
	\end{align}
	for all $r\in(0,R_0)$.
\end{lemma}
\begin{proof}
	Let us fix $r_0\in(0,R_0)$ and let us examine the limit
	\[
		\lim_{r\to r_0}\frac{H(r)-H(r_0)}{r-r_0}=\lim_{r\to r_0}\int_{S_1^+}y^a\frac{u^2(r z)-u^2(r_0 z)}{r-r_0}\ds(z),
	\]
	where $z=(x,y)\in S_1^+$. Since $u\in C^\infty(B_1^+)$ (see Remark \ref{rmk:smoothness}) we know that
	\[
		\lim_{r\to r_0}\frac{u^2(r z)-u^2(r_0 z)}{r-r_0}=2u(r_0 z)\nabla u(r_0 z)\cdot z.
	\]
	In order to pass to the limit under the integral sign, we make use of Lebesgue Dominated Convergence Theorem. In fact, from Lagrange Theorem and the regularity of the solution (see Lemma \ref{lemma:high_reg}) we derive that, for all $z\in S_1^+$ and $r\in (r_0/2,R_0)$,
	\[
		\left|\frac{u^2(r z)-u^2(r_0 z)}{r-r_0}\right|\leq C\sup_{B_{R_0}^+\setminus B_{r_0/2}^+}\abs{u}\bigg[\sup_{B_{R_0}^+\setminus B_{r_0/2}^+}\abs{\nabla_x u}+\sup_{B_{R_0}^+\setminus B_{r_0/2}^+}\left|y^a\frac{\partial u}{\partial y}\right|  \bigg],
	\]
	with $C>0$ depending on $R_0$ and $a$. We also recall that the weight $y^a$ is integrable, which allows us to conclude the proof of \eqref{eq:H'_th_1}. Finally, \eqref{eq:H'_th_2} is a consequence of \eqref{eq:H'_th_1}, \eqref{eq:R-N_int_th_2} and the definition of $\tilde{D}$ \eqref{eq:def_D_H}. The fact that $H'(r)\geq 0$ follows from the fact that $D$ and $G$ are nonnegative. The proof is thereby complete.

%	Thanks to the regularity of $u$, the fact that $H\in C^1$ and \eqref{eq:H'_th_1} easily come from direct computations and the Dominated Convergence Theorem, while \eqref{eq:H'_th_2} is a consequence of \eqref{eq:H'_th_1}, \eqref{eq:R-N_int_th_2} and the definition of $\tilde{D}$ \eqref{eq:def_D_H}.
\end{proof}

%\begin{remark}\label{H-mon} In view of the assumption $p\geq 2$, we infer from Lemma \ref{lemma:H'} that $H'(r)\geq 0$ for all $r\in (0,R_0)$.
%\end{remark}

\begin{lemma}\label{lemma:D'}
	We have that $\tilde{D}\in W^{1,1}_{\textup{loc}}(0,R_0)$ and
	\[
		\tilde{D}'(r)=\frac{2}{r^{N+a-1}}\int_{S_r^+}y^a\left(\frac{\partial u}{\partial \nnu}\right)^2\ds+\frac{2(1-a)}{pr}G(r)
	\]
	in a distributional sense and for a.e. $r\in(0,R_0)$.
\end{lemma}
\begin{proof}
	For any $r\in(0,R_0)$ we let
		\[
			E(r):=\int_{B_r^+}y^a\abs{\nabla u}^2\dxdy+\frac{2}{p}\int_{B_r'}F(u)\dx,
		\]
		so that
		\[
			\tilde{D}(r)=\frac{1}{r^{N+a-1}}E(r).
		\]
		First of all, it is easy to verify that $E\in L^1(0,R_0)$. Secondly, we compute its distributional derivative. By the coarea formula we have
		\begin{equation*}
			E'(r)=\int_{S_r^+}y^a\abs{\nabla u}^2\ds+\frac{2}{p}\int_{\partial B_r'}F(u)\dsi,
		\end{equation*}
		where $\dsi$ denotes the surface $N-1$ dimensional measure of $\R^N$. Again, it is easy to see that $E'\in L^1(0,R_0)$, thus implying that $E\in W^{1,1}(0,R_0)$. Now, by definition we have that $\tilde{D}\in W^{1,1}_{\textup{loc}}(0,R_0)$ and
		\[
			\tilde{D}'(r)=r^{-(N+a)}[rE'(r)-(N+a-1)E(r)].
		\]
		This, together with identity \eqref{eq:R-N_int_th_1}, yields the thesis.
\end{proof}

\begin{theorem}\label{thm:monot_N_tilde}
	The perturbed Almgren frequency function $\tilde{\mathcal{N}}$ belongs to $W^{1,1}_{\textup{loc}}(0,R_0)$ and it is nondecreasing in $(0,R_0)$, i.e. $\tilde{\mathcal{N}}'(r)\geq 0$ for all $r\in (0,R_0)$ in a distributional sense. In particular there exists
	\begin{equation*}
		k:=\lim_{r\to 0^+}\tilde{\mathcal{N}}(r)\in [0,\infty).
	\end{equation*}
\end{theorem}
\begin{proof}
	Combining Lemma \ref{lemma:H'} and Lemma \ref{lemma:D'} we have that
	\begin{align*}
		H^2(r)\tilde{\mathcal{N}}'(r)&=\tilde{D}'(r)H(r)-\tilde{D}(r)H'(r) \\
		&=\frac{2}{r^{2N+2a-1}}\left(\int_{S_r^+}y^a\left(\frac{\partial u}{\partial\nnu}\right)^2\ds\right)\left(\int_{S_r^+}y^au^2\ds\right)+\frac{2(1-a)}{pr}G(r)H(r)\\
		&-\frac{r}{2}(H'(r))^2+\frac{r}{2}\left(1-\frac{2}{p}\right)G(r)H'(r) \\
		&= \frac{2}{r^{2N+2a-1}}\left[\left(\int_{S_r^+}y^a\left(\frac{\partial u}{\partial\nnu}\right)^2\ds\right)\left(\int_{S_r^+}y^au^2\ds\right)-\left(\int_{S_r^+}y^au\,\frac{\partial u}{\partial\nnu}\ds \right)^2  \right] \\
		&+\frac{2(1-a)}{pr}G(r)H(r)+\left(1-\frac{2}{p}\right)G(r)H'(r)\geq 0,
	\end{align*}
	where the last inequality follows from Cauchy-Schwartz inequality, Lemma \ref{lemma:H'} and the assumption $a\in (-1,1)$.
\end{proof}

We now state a trace inequality on half-balls. The proof is rather standard, but we include it for the sake of completeness.

\begin{lemma}\label{lemma:trace_ineq}
	There exists a constant $C_1=C_1(N,a)>0$ such that
	\[
		\int_{B_r'}v^2\dx\leq C_1\left( r^{1-a}\int_{B_r^+}y^a\abs{\nabla v}^2\dxdy+r^{-a}\int_{S_r^+}y^a v^2\ds \right)
	\]
	for all $v\in H^{1,a}(B_r^+)$ and for all $r>0$.
\end{lemma}
\begin{proof}
	We provide the proof only for $v\in C^\infty(\overline{B_r^+})$ and conclude by standard density arguments. From the embedding $H^{1,a}(B_1^+)\hookrightarrow L^2(B_1')$ (see \cite[Theorem 2.3]{Nekvinda1993}) we know that there exists a constant $c=c(N,a)>0$ such that
	\begin{equation}\label{eq:trace_1}
		\int_{B_1'}v^2\dx\leq c\left( \int_{B_1^+}y^a\abs{\nabla v}^2\dxdy+\int_{B_1^+}y^a v^2\dxdy \right).
	\end{equation}
	We now prove that
	\begin{equation}\label{eq:trace_2}
		(N+a)\int_{B_1^+}y^av^2\dxdy\leq \int_{B_1^+}y^a\abs{\nabla v}^2\dxdy+\int_{S_1^+}y^a v^2\ds.
	\end{equation}
	By integration of the identity
	\[
		\dive(y^a v^2 z)=(N+a+1)y^av^2+2y^av\nabla v\cdot z
	\]
	in $B_1^+\cap \{(x,y)\colon y\geq \delta\}$, where $z=(x,y)$, for some small $\delta>0$, we obtain
	\begin{multline*}
		\int_{S_1^+\cap \{y\geq \delta\}}y^a v^2\ds+\delta^{1+a}\int_{B_1^+\cap\{y\geq \delta\}}v^2 \\
		=(N+a+1)\int_{B_1^+\cap\{y\geq \delta\}} y^av^2\dxdy+2\int_{B_1^+\cap\{y\geq \delta\}}y^av\nabla v\cdot z\dxdy.
	\end{multline*}
	Letting $\delta\to 0^+$, this implies
	\[
		\int_{S_1^+}y^a v^2\ds =(N+a+1)\int_{B_1^+} y^av^2\dxdy+2\int_{B_1^+}y^av\nabla v\cdot z\dxdy.
	\]
	Now, applying Young's inequality to the last term, using the fact that $\abs{z}\leq 1$ and rearranging the terms, we reach \eqref{eq:trace_2}. At this point, a combination of \eqref{eq:trace_1} and \eqref{eq:trace_2} yields
	\[
		\int_{B_1'}v^2\dx\leq \tilde{c}\left( \int_{B_1^+}y^a\abs{\nabla v}^2\dxdy+\int_{S_1^+}y^a v^2\ds \right),
	\]
	for some $\tilde{c}=\tilde{c}(N,a)>0$. By a scaling argument we may easily conclude the proof.
	
\end{proof}

We obtain as a consequence that the ``unperturbed'' frequency function $\mathcal{N}$ has a limit as $r \to 0^+$.

\begin{proposition}\label{prop:lim_N}
	Let $k\in[0,\infty)$ be as in Theorem \ref{thm:monot_N_tilde}. Then there exists the limit of $\mathcal{N}(r)$ as $r \to 0^+$ and $\lim_{r\to 0^+}\mathcal{N}(r)=k$.
\end{proposition}
\begin{proof}
	Let $r>0$ be sufficiently small. Trivially $\mathcal{N}(r)\leq \tilde{\mathcal{N}}(r)$. On the other hand, thanks to the boundedness of $u$, we have that
	\[
		\int_{B_r'}F(u)\dx\leq 2\max\{\lambda_-,\lambda_+\}\norm{u}_{L^\infty(B_1)}^{p-2}\int_{B_r'}u^2\dx.
	\]
	 Therefore, applying Lemma \ref{lemma:trace_ineq}, we deduce
	\begin{equation}\label{eq:lim_N_1}
		G(r)=\frac{1}{r^{N+a-1}}\int_{B_r'}F(u)\dx\leq C r^{1-a}(D(r)+H(r))
	\end{equation}
	 (with $C>0$ depending on $N,a,\lambda_\pm,p,\norm{u}_{L^\infty(B_1)},p$), which in turn implies
	 \[
	 	\tilde{\mathcal{N}}(r)\leq (1+Cr^{1-a})\mathcal{N}(r)+Cr^{1-a}.
	 \]
	 Hence
	 \[
	 	\frac{\tilde{\mathcal{N}}(r)-Cr^{1-a}}{1+C r^{1-a}}\leq \mathcal{N}(r)\leq \tilde{\mathcal{N}}(r)
	 \]
	 and the proof is thereby complete.
\end{proof}

\begin{corollary}\label{cor:bounds_H}
	Let $k=\lim_{r\to 0^+}\tilde{\mathcal{N}}(r)$. Then the following hold:
	\begin{enumerate}
		\item[(i)] there exists a constant $K_1>0$, depending on $N,a,\norm{u}_{L^\infty(B_1^+)},k$ such that
		\[
			H(r)\leq K_1 r^{2k}\quad\text{for all }r\in (0,R_0);
		\]
		\item[(ii)] for any $\delta>0$ there exists a constant $K_2=K_2(\delta)>0$, depending also on $N,a,u$ such that
		\[
			H(r)\geq K_2 r^{2k+\delta}\quad\text{for all }r\in (0,R_0).
		\]
	\end{enumerate}
\end{corollary}
\begin{proof}
	From Lemma \ref{lemma:H'} and Theorem \ref{thm:monot_N_tilde} we know that
	\[
		\frac{H'(r)}{H(r)}=\frac{2}{r}\left[\tilde{\mathcal{N}}(r)+\left(1-\frac{2}{p}\right)\frac{G(r)}{H(r)} \right]\geq \frac{2}{r}k.
	\]
	By integrating the inequality above in $(r,R_0)$, thanks to the boundedness of $u$, we conclude the proof of point $(i)$. On the other hand, we recall that
	\[
		\frac{H'(r)}{H(r)}=\frac{2}{r}\left[\mathcal{N}(r)+\frac{G(r)}{H(r)}  \right].
	\]
	Keeping \eqref{eq:lim_N_1} in mind, we obtain
	\[
		\frac{H'(r)}{H(r)}\leq\frac{2}{r}\left[\mathcal{N}(r)+Cr^{1-a}(\mathcal{N}(r)+1)  \right].
	\]
	Hence, since $\lim_{r\to 0^+}\mathcal{N}(r)=k$, for any $\delta>0$ there exists $r_0=r_0(\delta)>0$ such that, for all $r\in(0,r_0)$
	\[
		\frac{H'(r)}{H(r)}\leq \frac{2k +\delta }{r}.
	\]
	Integrating the inequality above and taking into account the boundedness of $u$, we obtain that
	\[
		H(r)\geq C_2 r^{2k+\delta}\quad\text{for all }r\in(0,r_0).
	\]
	Thanks to continuity and positiveness of $H$ we may conclude the proof of $(ii)$.
\end{proof}

The previous result readily implies an estimate on the growth rate of the solution $u$ near the origin.

\begin{corollary}\label{cor:growth_estim}
	Let $x_0\in B_{R_0/2}'$ and let $k=\lim_{r\to 0^+}\tilde{\mathcal{N}_{x_0}}(u,r)$. Then there exists $\tilde{C}>0$ depending on $N$, $a$, $\norm{g}_{L^\infty(S_1^+)}$ and $k$ (independent of $x_0$) such that
	\[
	\abs{u(x+x_0,0)}\leq \tilde{C}\abs{x}^k\quad\text{for all }x\in B_{R_0/2}'.
	\]
\end{corollary}
\begin{proof}
	Let $x\in B_{R_0/2}'$ and let $r:=\abs{x}$. Thanks to Lemma \ref{lemma:subharm} and Lemma \ref{lemma:mean_value_ineq} we have that
	\[
	\abs{u(x+x_0,0)}\leq \frac{\alpha_{N,a}}{r^{N+1+a}}\int_{B_r^+(x+x_0)}y^a\abs{u}\dxdy.
	\]
	Moreover
	\[
	\int_{B_r^+(x+x_0)}y^a\abs{u}\dxdy\leq \int_{B_{2r}^+(x_0)}y^a\abs{u}\dx.
	\]
	From the Cauchy-Schwartz inequality we deduce that
	\[
	\int_{B_{2r}^+(x_0)}y^a\abs{u}\dxdy\leq C_1r^{\frac{1}{2}(N+1+a)}\left( \int_{B_{2r}^+(x_0)}y^a u^2\dxdy \right)^{1/2},
	\]
	for some $C_1>0$ depending on $N,a$ and independent of $x_0$. Gathering in chain all the previous inequalities we obtain that
	\begin{equation}\label{eq:growth_estim_1}
		\abs{u(x+x_0,0)}\leq\frac{C_2}{r^{\frac{1}{2}(N+1+a) }}\left( \int_{B_{2r}^+(x_0)}y^a u^2\dxdy \right)^{1/2},
	\end{equation}
	with $C_2>0$ again depending only on $N,a$. Now, Corollary \ref{cor:bounds_H}, $(i)$, implies that
	\[
	\int_{B_{2r}^+(x_0)}y^a u^2\dxdy=\int_0^{2r}\left(\int_{S_\rho^+(x_0)}y^au^2\ds \right)\mathrm{d}\,\rho\leq C_3 r^{N+1+a+2k},
	\]
	where $C_3>0$ depends on $N,a,\norm{g}_{L^\infty(S_1^+)},k$. In fact, the proof of Corollary \ref{cor:bounds_H}, $(i)$, can be trivially adapted to the function $H_{x_0}(u,r)$ with the same constant $K_1$. The proof may now be concluded by combining the inequality above with \eqref{eq:growth_estim_1}.
\end{proof}

We now introduce the following functional of Monneau type
\begin{equation}\label{eq:def_Monn}
M_\gamma^{x_0}(v,w,r):=\frac{1}{r^{2\gamma}}H_{x_0}(v-w,r),
\end{equation}
defined for $v,w\in H^{1,a}(B_1^+)$, $x_0\in B_1'$, $r\in (0,1)$ and $\gamma\geq 0$, with $H_{x_0}$ as in \eqref{eq:def_D_H}. As indicated before, we will drop the index $x_0$ in the notation for $M_\gamma^{x_0}$ when $x_0=0$. In the following proposition we prove the monotonicity of this functional with respect to $r\in (0,R_0/2)$ when the solution $u$ and a polynomial in $\mathds{P}_k^a$ are ``compared''.

\begin{proposition}\label{prop:M'}
	Let $k=\lim_{r\to 0^+}\tilde{\mathcal{N}}(r)$ and let $p_k\in \mathds{P}_k^a$ be an $a$-harmonic homogeneous polynomial of degree $k$, even in $y$, i.e. respectively satisfying \eqref{eq:polynom3}, \eqref{eq:polynom2}, \eqref{eq:polynom1}. Then $M_k(u,p_k,\cdot)\in C^1(0,R_0/2)$ and
	\[
	\frac{\mathrm{d}}{\mathrm{d}r}M_k(u,p_k,r)\geq -\hat{C}r^{k(p-2)-a},
	\]
	for some $\hat{C}>0$ depending on $N$, $a$, $\norm{g}_{L^\infty(S_1^+)}$ and $k$.
\end{proposition}
\begin{proof}
	Let $r\in(0,R_0/2)$. Since
	\begin{align}
	\frac{\mathrm{d}}{\d r}M_k(u,p_k,r)&=\frac{\d}{\d r}(r^{-2k}H(u-p_k,r)) \notag\\
	&=-2k r^{-2k-1}H(u-p_k,r)+r^{-2k}\frac{\d}{\d r}H(u-p_k,r),\label{eq:M'_1}
	\end{align}
	we thus need to compute the latter. By integrating the identity
	\[
	\dive(y^a u\nabla p_k)=y^a\nabla u\cdot\nabla p_k+u\dive(y^a\nabla p_k)
	\]
	in $B_r^+$ we obtain
	\begin{equation}\label{eq:M'_2}
	\int_{B_r^+}y^a\nabla u\cdot\nabla p_k\dxdy=\int_{S_r^+}y^a u\frac{\partial p_k}{\partial \nnu}\ds,
	\end{equation}
	since $p_k$ is $a$-harmonic and $y^a\frac{\partial p_k}{\partial y}=0$ on $B_r'$. Moreover, on $S_r^+$, we have
	\[
	\frac{\partial p_k}{\partial \nnu}=\frac{1}{r}\nabla p_k\cdot z=\frac{k}{r}p_k
	\]
	by the homogeneity of $p_k$. Combining this fact with \eqref{eq:M'_2} we have that
	\begin{equation}\label{eq:M'_3}
	r\int_{B_r^+}y^a\nabla u\cdot\nabla p_k\dxdy=k\int_{S_r^+}y^a u p_k\ds.
	\end{equation}
	On the other hand, we have that the function $u-p_k$ weakly satisfies
	\begin{equation*}
	\begin{bvp}
	-\dive(y^a\nabla (u-p_k))&=0, &&\text{in }B_r^+, \\
	-\lim_{y\to  0^+}y^a\frac{\partial}{\partial y}(u-p_k)&=\lambda_-(u^-)^{p-1}-\lambda_+(u^+)^{p-1}, &&\text{on }B_r',
	\end{bvp}
	\end{equation*}
	that, multiplied by $u-p_k$ and integrated gives that
	\begin{multline}\label{eq:M'_4}
	\int_{S_r^+}y^a(u-p_k)\frac{\partial }{\partial\nnu}(u-p_k)\ds
	=\int_{B_r^+}y^a\abs{\nabla (u-p_k)}^2\dxdy\\
	+\int_{B_r'}(\lambda_-(u^+)^p+\lambda_+(u^+)^p)\dx+\int_{B_r'}(\lambda_-(u^-)^{p-1}-\lambda_+(u^+)^{p-1})p_k\dx.
	\end{multline}
	Let us denote
	\[
		I(u,p_k,r):=\frac{1}{r^{N+a-1}}\int_{B_r'}(\lambda_-(u^-)^{p-1}-\lambda_+(u^+)^{p-1})p_k\dx.
	\]
	Reasoning as in the proof of Lemma \ref{lemma:H'}, we have that
	\[
	\frac{\d}{\d r}H(u-p_k,r)=\frac{2}{r^{N+a}}\int_{S_r^+}y^a(u-p_k)\frac{\partial }{\partial\nnu}(u-p_k)\ds.
	\]
	Therefore, combining it with \eqref{eq:M'_4} and \eqref{eq:M'_3} we deduce that
	\begin{align*}
	\frac{\d}{\d r}H(u-p_k,r)=\frac{2}{r}[&D(u-p_k,r)+G(u,r) +I(u,p_k,r)] \\
	=\frac{2}{r}\bigg[&D(u,r)+D(p_k,r)+G(u,r) \\
	&-\frac{2}{r^{N+a-1}}\int_{B_r^+}y^a\nabla u\cdot\nabla p_k\dxdy+I(u,p_k,r)\bigg] \\
	=\frac{2}{r}\bigg[&D(u,r)+D(p_k,r)+G(u,r)-\frac{2k}{r^{N+a}}\int_{S_r^+}y^a u p_k\ds+I(u,p_k,r)\bigg].
	\end{align*}
	In addition, since $p_k$ is homogeneous, then $D(p_k,r)=k H(p_k,r)$; hence
	\begin{equation}\label{eq:M'_5}
	\frac{\d}{\d r}H(u-p_k,r)= \frac{2}{r}\left[D(u,r)+k H(p_k,r)+G(u,r)-\frac{2k}{r^{N+a}}\int_{S_r^+}y^a u p_k\ds+I(u,p_k,r)\right].
	\end{equation}
	Finally, let us compute the other term in \eqref{eq:M'_1}
	\begin{equation}\label{eq:M'_6}
	H(u-p_k,r)=H(u,r)+H(p_k,r)-\frac{2}{r^{N+a}}\int_{S_r^+}y^a u p_k\ds.
	\end{equation}
	Putting together \eqref{eq:M'_5} and \eqref{eq:M'_6} with \eqref{eq:M'_1} we obtain that $M_k(u,p_k,\cdot)\in C^1(0,R_0/2)$ and
	\begin{align}
		\frac{\d}{\d r}M_k(u,p_k,r)=&\frac{2}{r^{2k+1}}\left[ \tilde{D}(u,r)-kH(u,r)+\left(1-\frac{2}{p}\right)G(u,r)+I(u,p_k,r) \right]\nonumber \\
		=&\frac{2}{r^{2k+1}}\left[ \tilde{D}(u,r)-kH(u,r)+\left(1-\frac{2}{p}\right)G(u,r)\right]\label{eq:M'_7}\\
		&+\frac{2}{r^{N+a+2k}}\int_{B_r'}(\lambda_-(u^-)^{p-1}-\lambda_+(u^+)^{p-1})p_k\dx.\nonumber
	\end{align}
	From Corollary \ref{cor:growth_estim} we have that
	\[
		\norm{u}_{L^\infty(B_r')}\leq \tilde{C}r^k,
	\]
	while, being $p_k$ an homogeneous polynomial of degree $k$,
	\[
		\norm{p_k}_{L^\infty(B_r')}\leq \tilde{C}_1 r^k
	\]
	for some $\tilde{C}_1>0$. Therefore
	\[
		\left|\frac{2}{r^{N+a+2k}}\int_{B_r'}(\lambda_-(u^-)^{p-1}-\lambda_+(u^+)^{p-1})p_k\dx\right|\leq \hat{C}r^{k(p-2)-a}
	\]
	for some $\hat{C}>0$ depending on $N$, $a$, $\norm{g}_{L^\infty(S_1^+)}$ and $k$. In view of the inequality above and \eqref{eq:M'_7}, together with  the fact that $\tilde{\mathcal{N}}(r)\geq k$ and that $p\geq 2$, the proof is complete.

\end{proof}

\begin{corollary}\label{cor:lim_monn}
	There exists $\lim_{r\to 0^+}M_k(u,p_k,r)=:M_k(u,p_k,0^+)\in[0,+\infty)$. Moreover
	\begin{equation}\label{eq:lim_monn_th}
		M_k(u,p_k,0^+)\leq M_k(u,p_k,r)+\hat{C}_1r^{k(p-2)+1-a},\quad\text{for all }r\in(0,R_0/2),
	\end{equation}
	where $\hat{C}_1:=\hat{C}(k(p-2)+1-a)^{-1}$ and $\hat{C}>0$ is as in Proposition \ref{prop:M'}.
\end{corollary}
\begin{proof}
	From Proposition \ref{prop:M'} we can derive that
	\begin{equation}\label{eq:lim_monn_1}
		\frac{\d}{\d r}\left[ M_k(u,p_k,r)+\frac{\hat{C}r^{k(p-2)+1-a}}{k(p-2)+1-a} \right]\geq 0.
	\end{equation}
	Hence the  limit
	\[
		\lim_{r\to 0^+} \left[M_k(u,p_k,r)+\frac{\hat{C}r^{k(p-2)+1-a}}{k(p-2)+1-a}\right]
	\]
exists and it is finite.
	Since $a\in(0,1)$ and $p\geq 2$, we know that $r^{k(p-2)+1-a}$ vanishes as $r\to 0^+$, and therefore
	\[
		\lim_{r\to 0^+} \left[M_k(u,p_k,r)+\frac{\hat{C}r^{k(p-2)+1-a}}{k(p-2)+1-a}\right]=	\lim_{r\to 0^+} M_k(u,p_k,r)=:M_k(u,p_k,0^+).
	\]
	In addition $M_k(u,p_k,r)\geq 0$ for all $r\in (0,R_0/2)$ and so $M_k(u,p_k,0^+)\geq 0$ as well. Finally, by integration in $(0,r)$ of \eqref{eq:lim_monn_1} we obtain \eqref{eq:lim_monn_th} and we conclude the proof.	
\end{proof}

\section{Blow-up analysis}\label{sec:blow_up}

This section is devoted the blow-up analysis of a particular normalization of the solution, which will lead to the proof of Theorem \ref{thm:blow_up_homo}.

Hereafter, we fix $\tilde{R}\in(0,R_0/2)$. In this section we denote, for $z\in B_1^+$ and $r\in(0,\tilde{R})$,
\begin{equation}\label{eq:u_rescaled_H}
	u_r(z):=\frac{u(rz)}{\sqrt{H(r)}}.
\end{equation}

The first step is the blow-up analysis for the function $u_r$ as in \eqref{eq:u_rescaled_H}.

\begin{theorem}\label{thm:blow_up}
	Let $k=\lim_{r\to 0^+}\mathcal{N}(r)$ and $u_r$ as in \eqref{eq:u_rescaled_H}. Then $k\in \N\cup \{0\}$ and for any sequence $r_n\to 0^+$ there exists a subsequence $r_{n_i}\to 0^+$ and a function $\tilde{u}\in \mathds{P}_k^a$, $\tilde{u}\not\equiv 0$, such that
		\begin{gather}
			u_{r_{n_i}}\to \tilde{u} \quad\text{in }H^{1,a}(B_1^+)~\text{and }C^{0,\alpha}(\overline{B_1^+}) \label{eq:blow_up_th_1} \\
			u_{r_{n_i}}\to\tilde{u}\quad\text{in }C^{1,\alpha}(B_1') \label{eq:blow_up_th_2}
		\end{gather}
		as $i\to \infty$;
\end{theorem}
\begin{proof}
	Let $r_n\to 0^+$. First, we have that $u_{r_n}\in H^{1,a}(B_1^+)$ and
	\[
		\int_{S_1^+}y^au_{r_n}^2\ds=H(u_{r_n},1)=1.
	\]
	Moreover
	\[
		\int_{B_1^+}y^a\abs{\nabla u_{r_n}}^2\dxdy=\mathcal{N}(u_{r_n},1)=\mathcal{N}(u,r_n)\leq \tilde{\mathcal{N}}(u,\tilde{R}).
	\]
	Therefore, thanks to the Poincar\'e-type inequality Lemma \ref{lemma:poincare}, we know that $\{u_{r_n}\}_n$ is bounded in $H^{1,a}(B_1^+)$. Thus there exists a subsequence $r_{n_i}\to 0^+$ and a function $\tilde{u}\in H^{1,a}(B_1^+)$ such that
	\[
		u_{r_{n_i}}\rightharpoonup \tilde{u}\quad\text{weakly in }H^{1,a}(B_1^+)
	\]
	as $i\to\infty$. Due to the compact embedding, $u_{r_{n_i}}\to \tilde{u}$ strongly in $L^{2,a}(S_1^+)$, which implies that $\norm{\tilde{u}}_{L^{2,a}(S_1^+)}=1$ and therefore $\tilde{u}\not\equiv 0$. Next, we have that $u_r\in H^{1,a}(B_1^+)$ satisfies
	\[
		\int_{B_1^+}y^a\nabla u_r\cdot\nabla \phi\dxdy=r^{1-a}(H(r))^{\frac{p-2}{2}}\int_{B_1'}(\lambda_-(u_r^-)^{p-1}-\lambda_+(u_r^+)^{p-1})\phi\dx
	\]
	for all $\phi\in C_c^\infty(B_1^+\cup B_1')$. Hence, for $\phi\in C_c^\infty(B_1^+\cup B_1')$ we have
	\begin{equation}\label{eq:blow_up_1}
		\left| \int_{B_1^+}y^a\nabla u_{r_{n_i}}\cdot\nabla \phi\dxdy \right|\leq C r_{n_i}^{1-a}(H(r_{n_i}))^{\frac{p-2}{2}}\int_{B_1'}\abs{u_{r_{n_i}}}^{p-1}\dx,
	\end{equation}
	where $C=2\max\{\lambda_-,\lambda_+\}\norm{\phi}_{L^\infty(B_1^+)}$. We now observe that, thanks to Corollary \ref{cor:growth_estim},
		\[
		\abs{u_{r_{n_i}}(x,0)}^{p-1}\leq \frac{Cr_{n_i}^{k(p-1)}}{(H(r_{n_i}))^{\frac{p-1}{2}}},
		\]
		for another constant $C>0$. Combining this with \eqref{eq:blow_up_1} we obtain that
		\[
		\left| \int_{B_1^+}y^a\nabla u_{r_{n_i}}\cdot\nabla \phi\dxdy \right|\leq \frac{C r_{n_i}^{k(p-1)+1-a}}{\sqrt{H(r_{n_i})}}.
		\]
		up to renaming the constant $C$. Now, by choosing $\delta=1-a$ in Corollary \ref{cor:bounds_H}, $(ii)$, and plugging the estimate on $H$ in the previous equation, we derive that
		\[
		\left| \int_{B_1^+}y^a\nabla u_{r_{n_i}}\cdot\nabla \phi\dxdy \right|\leq C r_{n_i}^{k(p-2)+\frac{1-a}{2}},
		\]
		possibly renaming again the constant $C$. Thus, passing to the limit for $i\to\infty$ yields that $\tilde{u}$ is $y^a$-harmonic in $B_1^+\cup B_1'$. Thanks to the regularity of $u_r$ in H\"older spaces (see Lemma \ref{lemma:regularity} and Lemma \ref{lemma:high_reg}) we can easily deduce that
		\[
			u_{r_{n_i}}\to \tilde{u}\quad\text{in }C^{0,\alpha}(\overline{B_1^+})\text{ as }i\to\infty,
		\]
		as well as
		\[
				\nabla_xu_{r_{n_i}}\to\nabla_x \tilde{u}\quad\text{and}\quad y^a\frac{\partial u_{r_{n_i}}}{\partial_y}\to y^a\frac{\partial u_{r_{n_i}}}{\partial_y}\quad\text{in }C^{0,\alpha}(\overline{B_1^+}),\quad\text{as }i\to\infty.
		\]
		An integration by parts implies that $\lVert u_{r_{n_i}}\rVert_{H^{1,a}(B_1^+)}\to \norm{\tilde{u}}_{H^{1,a}(B_1^+)}$. Hence \eqref{eq:blow_up_th_1} and \eqref{eq:blow_up_th_2} are proved.
	
	Now we can pass to the limit in the Almgren frequency function and obtain that, for any $R\in(0,1)$
	\[
		k=\lim_{i\to\infty}\mathcal{N}(u,R r_{n_i})=\lim_{i\to\infty}\mathcal{N}(u_{r_{n_i}},R)=\mathcal{N}(\tilde{u},R),
	\]
	which implies that $\tilde{u}$ is homogeneous of degree $k$ (for the proof of this fact see e.g. \cite[Proposition 4.6]{Sire2020}).  Therefore, the even reflection of $\tilde{u}$ through $\{y=0\}$ is an $a$-harmonic homogeneous polynomial of degree $k$. But such a polynomial must be smooth (see e.g. \cite[Theorem 1.1]{Sire2019}) and so we know that $k$ must be a nonnegative integer and that $\tilde{u}\in\mathds{P}_k^a$.
\end{proof}

Next, we prove a nondegeneracy result of the solution at any boundary point.

\begin{proposition}\label{prop:nondegeneracy}
	There exists two constants $C_1,C_2>0$ such that the following estimates from above and below hold:
	\begin{equation}\label{eq:nondeg_th_1}
		\abs{u(z)}\leq C_1\abs{z}^k\quad\text{for all }z\in B_r^+
	\end{equation}
	and
	\begin{equation}\label{eq:nondeg_th_2}
		\sup_{S_r^+}\abs{u}\geq C_2r^{k}
	\end{equation}
	for all $r\in (0,\tilde{R})$. Moreover, there exists $\lim_{r\to 0^+}r^{-2k}H(u,r)\in (0,+\infty)$.
\end{proposition}
\begin{proof}
	In order to prove \eqref{eq:nondeg_th_1}, let us assume by contradiction that there exists a sequence $r_n\to 0^+$ such that
	\[
	\sup_{ z\in S_{r_n}^+}\abs{u(z)}\geq n r_n^{k}.
	\]
	Together with Corollary \ref{cor:bounds_H} and a change of variable this means that
	\[
	\sup_{S_1^+}\abs{u_{r_n}}^2\geq K_1 n \int_{S_1^+}y^a u_{r_n}^2\ds.
	\]
	But in view of Theorem \ref{thm:blow_up} this gives rise to a contradiction for $n$ sufficiently large.	
	
	Now let us prove \eqref{eq:nondeg_th_2}. We observe that \eqref{eq:nondeg_th_2} is equivalent to
	\begin{equation}\label{eq:nondeg_1}
		r^{-2k}\sup_{S_r^+}\abs{u}^2\geq C_2^2>0.
	\end{equation}
	By explicit computation it is possible to check that
	\[
		\int_{S_r^+}y^a\ds = \bar{C}r^{N+a},
	\]
	for some $\bar{C}>0$ (depending on $N$ and $a$), thus implying that	
	\begin{equation*}
		r^{-2k}H(u,r)\leq \bar{C}r^{-2k}\sup_{S_r^+}\abs{u}^2.
	\end{equation*}
	 So let us contradict \eqref{eq:nondeg_1} and assume that, up to a subsequence
	\[
		0\leq r^{-2k}H(u,r)\leq \bar{C}r^{-2k}\sup_{S_r^+}\abs{u}^2 \to 0,
	\]
	and so that
	\begin{equation}\label{eq:nondeg_2}
		 H(u,r)=o(r^{2k})\quad\text{as }r\to 0^+.
	\end{equation} Let $\tilde{u}\in H^{1,a}(B_1^+)$ be, as in Theorem \ref{thm:blow_up}, the blow-up limit of $u_r(z)=[H(u,r)]^{-1/2}u(rz)$, possibly passing to another subsequence as $r\to 0^+$. By definition of the Monneau-type functional, we have
	\begin{equation}\label{eq:nondeg_3}
		M_k(u,\tilde{u},r)=r^{-2k}\left(H(u,r)+H(\tilde{u},r)-\frac{2}{r^{N+a}}\int_{S_r^+}y^a u \tilde{u}\ds  \right).
	\end{equation}
	Thanks to homogeneity of $\tilde{u}$, we notice that
	\begin{equation}\label{eq:nondeg_4}
		H(\tilde{u},r)=r^{2k}H(\tilde{u},1).
	\end{equation}
	On the other hand, by H\"older's inequality, \eqref{eq:nondeg_2} and \eqref{eq:nondeg_4}
	\begin{equation*}
		\left|r^{-N-a-2k}\int_{S_r^+}y^a u\tilde{u}\dxdy  \right|\leq \sqrt{r^{-2k}H(u,r)}\sqrt{r^{-2k}H(\tilde{u},r)}=o(1)
	\end{equation*}
	as $r\to 0^+$. Inserting this, together  with \eqref{eq:nondeg_2} and \eqref{eq:nondeg_4}, in \eqref{eq:nondeg_3} we obtain
	\begin{equation*}
		M_k(u,\tilde{u},r)\to H(\tilde{u},1)
	\end{equation*}
	as $r\to 0^+$. Then, by monotonicity of $M_k(u,\tilde{u},r)$, proved in Proposition \ref{prop:M'}, we have that
	\begin{equation*}
		r^{-2k}H(\tilde{u},r)=H(\tilde{u},1)\leq r^{-2k}H(u-\tilde{u},r)+\hat{C}_1 r^{k(p-2)+1-a}.
	\end{equation*}
	We infer
	\begin{equation}\label{eq:nondeg_5}
		\frac{1}{r^{N+a}}\int_{S_r^+}y^a(u^2-2u\tilde{u})\ds\geq -\hat{C}_1r^{kp+1-a}
	\end{equation}
	for the chosen sequence $r\to 0^+$. But, after rescaling and dividing by $\sqrt{r^{2k} H(u,r)}$, thanks again to the homogeneity of $\tilde{u}$, \eqref{eq:nondeg_5} becomes
	\begin{equation*}
		\int_{S_1^+}y^a(\sqrt{r^{-2k}H(u,r)}u_r^2-2u_r\tilde{u})\ds\geq -\frac{\hat{C}_1 r^{k(p-1)+1-a}}{\sqrt{H(u,r)}}.
	\end{equation*}
	Now, let us apply Corollary \ref{cor:bounds_H}, $(ii)$, with $\delta=1-a$ to find that there exists $K_2>0$ such that $H(u,r)\geq K_2 r^{2k+\frac{1-a}{2}}$. Continuing with the inequality above we get
	\[
		\int_{S_1^+}y^a(\sqrt{r^{-2k}H(u,r)}u_r^2-2u_r\tilde{u})\ds\geq -\hat{C}_1 r^{k(p-2)+\frac{1-a}{2}}.
	\]
	Passing to the limit as $r\to 0^+$ we obtain
	\[
		-2H(\tilde{u},1)=-2\int_{S_1^+}y^a \tilde{u}^2\ds\geq 0,
	\]	
	thus giving rise to a contradiction since $H(\tilde{u},1)>0$.  If not, in fact, the unique continuation principle would be violated, as $H(\tilde{u},1)=r^{-2k}H(\tilde{u},r)$ for all $r\in (0,1)$. Then there exists a constant $C>0$ such that
	\begin{equation*}
		C\leq r^{-2k}H(u,r)\leq \bar{C}r^{-2k}\sup_{S_r^+}\abs{u}^2
	\end{equation*}
	for all $r\in (0,1)$ and this proves \eqref{eq:nondeg_th_2}. The second part of the thesis comes from the inequality above and Corollary \ref{cor:bounds_H}, $(i)$.
	
\end{proof}

Next, we use once more  the monotonicity of the Monneau functional  \eqref{eq:def_Monn} to uniquely identify the blow-up limit and to prove Theorem \ref{thm:blow_up_homo}.

\begin{proof}[Proof of Theorem \ref{thm:blow_up_homo}]
	Let $r_n\to 0^+$. By Theorem \ref{thm:blow_up} there exists a subsequence $r_{n_i}\to 0^+$ and an $a$-harmonic homogeneous polynomial $\tilde{u}$, even in $y$, such that
	\[
		u_{r_{n_i}}\to \tilde{u} \quad\text{in }H^{1,a}(B_1^+),~C^{0,\alpha}(\overline{B_1^+})~\text{and } C^{1,\alpha}(B_1')
	\]
	as $i\to \infty$. Therefore, if we let $\gamma:=\lim_{r\to 0^+}r^{-2k}H(u,r)\in (0,+\infty)$, whose existence is ensured by Proposition \ref{prop:nondegeneracy}, we have that
	\[
		u_{r_{n_i}}^{k}\to u_0:=\sqrt{\gamma}\,\tilde{u} \quad\text{in }H^{1,a}(B_1^+),~C^{0,\alpha}(\overline{B_1^+})~\text{and } C^{1,\alpha}(B_1')
	\]
	as $i\to\infty$, where
	\[
		u_r^k(z):=\frac{u(rz)}{r^k}.
	\]
	 By Corollary \ref{cor:lim_monn} the functional $M_k(u,u_0,r)$ admits a limit as $r\to 0^+$, and therefore
	\[
		\lim_{r\to 0^+}M_k(u,u_0,r)=\lim_{i\to\infty}M_k(u,u_0,r_{n_i})=\lim_{i\to\infty}M_k(u_{r_{n_i}}^k,u_0,1)=0.
	\]
	Now we prove that the limit does not depend on the choice of the subsequence. Let $r_{n_i}\to 0^+$ another subsequence and $\tilde{u}_0$ another possbile limit. Again, $\tilde{u}_0$ will be an $a$-harmonic homogeneous polynomial of degree $k$ and even in $y$. Hence, analogously
	\[
		\lim_{r\to 0^+}M_k(u,\tilde{u}_0,r)=\lim_{i\to\infty}M_k(u,\tilde{u}_0,r_{n_i})=\lim_{i\to\infty}M_k(u_{r_{n_i}}^k,\tilde{u}_0,1)=0.
	\]
	 Furthermore
	\[
		\int_{S_1^+}y^a(u_0-\tilde{u}_0)^2\dxdy=M_k(u_0,\tilde{u}_0,r)\leq 2M_k(u,u_0,r)+ 2M_k(u,\tilde{u}_0,r)\to 0
	\]
	as $r\to 0^+$ and so $u_0=\tilde{u}_0$. The proof is concluded.
	
\end{proof}

\begin{remark}\label{rmk:homo_blow_up}
	First of all we point out that the functions $\tilde{u}$ and $u_0$, defined respectively in Theorem \ref{thm:blow_up} and Theorem \ref{thm:blow_up_homo}, are such that $u_0=\sqrt{\gamma}\,\tilde{u}$, where $\gamma=\lim_{r\to 0^+}r^{-2k}H(u,r)~>~0$.
	
	Moreover, since $u_0$ is an $a$-harmonic polynomial in $\R^{N+1}_+$ and since $\lim_{y\to 0^+}y^a\frac{\partial u_0}{\partial y}~=~0$, then $u_0$ cannot vanish everywhere on the thin space $\{y=0\}$, because otherwise its trivial extension to the whole $\R^{N+1}$ would violate the unique continuation principle for $A_2$-weighted elliptic equations, see \cite[Corollary 1.4]{Tao2008} (see also \cite[Lemma 5.2]{MR3916700}).
\end{remark}

\begin{remark}\label{rmk:def_blow_up_lim}
	We recall that
	\begin{align*}
		\mathcal{Z}(u)&=\{ x_0\in B_1'\colon u(x_0,0)=0 \}, \\
		\mathcal{Z}_k(u)&=\{x_0\in \mathcal{Z}(u)\colon\mathcal{N}_{x_0}(u,0^+)=k \}.
	\end{align*}
	Thanks to Theorem \ref{thm:blow_up_homo}, for any $x_0\in \mathcal{Z}(u)$ there exists a unique integer $k\geq 0$ and a unique $a$-harmonic polynomial $p_k^{x_0}\in \mathds{P}_k^a$, homogeneous of degree $k$ such that
	\[
		u_r^{k,x_0}(z):=\frac{u(rz+x_0)}{r^{k}}\to p_k^{x_0}(z)\quad\text{in }H^{1,a}(B_1^+).
	\]
	We refer to this polynomial as the \emph{blow-up limit} of $u$ at the point $x_0$. In particular, the fact that $u(x_0,0)=0$ readily implies that $k\geq 1$.
\end{remark}

We conclude the section by proving the continuous dependence of the blow-up limits on the nodal points. This result plays a pivotal role in the analysis of the structure of the free boundary.

\begin{theorem}\label{thm:cont_blow_up}
	Let $k \in\N$, $k\geq 1$. Then the map
	\begin{align*}
		\mathcal{Z}_k(u) & \longrightarrow \mathds{P}_k^a \\
		x_0 &\longmapsto p_k^{x_0}
	\end{align*}
	is continuous, where $\mathcal{Z}_1(u)=\mathcal{R}(u)$ and $p_k^{x_0}$ as in Remark \ref{rmk:def_blow_up_lim}. Moreover for any compact $K\sub \mathcal{Z}_k(u)$
	\begin{equation*}
		\norm{ u_r^{k,x_0}-p_k^{x_0}}_{L^\infty(B_1^+)}\to 0\quad\text{as }r\to 0^+,~\text{uniformly for }x_0\in K.
	\end{equation*}
	In particular there exists a modulus of continuity $\sigma_K(t)\geq 0$, $\sigma_K(0^+)=0$ such that
	\[
		\abs{u(z)-p_k^{x_0}(z-x_0)}\leq \sigma_K(\abs{z-x_0})\abs{z-x_0}^k,
	\]
	for any $x_0\in K$.
\end{theorem}
\begin{proof}
	First we observe that $\mathds{P}_k^a$ is a subset of finite dimensional space, namely the space of homogeneous polynomial of degree $k$, and so all the norms are equivalent: we choose to endow this space with the $L^{2,a}(S_1^+)$-norm. Thanks to Theorem \ref{thm:blow_up_homo}, for any $x_0\in \mathcal{Z}_k(u)$ and any $\epsilon>0$ there exists $r_\epsilon=r_\epsilon(x_0)\in (0,\epsilon)$ such that
	\begin{equation*}
		M_k^{x_0}(u,p_k^{x_0},r)=r^{-N-a-2k}\int_{S_{r}^+}y^a(u(z+x_0)-p_k^{x_0}(z))^2\ds <\epsilon
	\end{equation*}
	for all $r<r_\epsilon$.	Moreover, due to $\alpha$-H\"older continuity of $u$ and the previous inequality, there exists $\delta_\epsilon=\delta_\epsilon(x_0,k,\alpha)>0$ such that
	\begin{equation}\label{eq:mod_cont_6}
		M_k^{\bar{x}}(u,p_k^{x_0},r)=r^{-N-a-2k}\int_{S_{r}^+}y^a(u(z+\bar{x})-p_k^{x_0}(z))^2\ds <2\epsilon,
	\end{equation}
	for any $\bar{x}\in \mathcal{Z}_k(u)$ satisfying $\abs{\bar{x}-x_0}<\delta_\epsilon$ and for any $r<r_\epsilon$. Then, the almost monotonicity of $M_k^{\bar{x}}(u,p_k^{x_0},r)$ (see Corollary \ref{cor:lim_monn}) yields
	\begin{multline}\label{eq:mod_cont_1}
		\int_{S_1^+}y^a(p_k^{\bar{x}}-p_k^{x_0})^2\ds=\lim_{r\to 0^+}M_k^{\bar{x}}(u,p_k^{x_0},r)\\\leq M_k^{\bar{x}}(u,p_k^{x_0},r)+\hat{C}_1 r^{k(p-2)+1-a}<3\epsilon
	\end{multline}
	and for any $\bar{x}\in \mathcal{Z}_k(u)$ such that $\abs{\bar{x}-x_0}<\delta_\epsilon$ and for any $r<r_\epsilon$ (up to taking a smaller $r_\epsilon$, chosen independently of $\bar{x}$). This proves the first part of the statement.
	
	Now, let us fix a compact $K\sub \mathcal{Z}_k(u)$ and let $x_0\in K$. From now until the end of the proof, we denote by $C$ a positive constant, whose value may change from line to line, that is independent of $\epsilon$ and of the choice of $\bar{x}$. From \eqref{eq:mod_cont_6} we deduce that
	\[
	\int_{S_{r}^+}y^a(u(z+\bar{x})-p_k^{x_0}(z))^2\ds \leq 2\epsilon r^{N+a+2k}
	\]
	which, integrated in $r$ and up to a change of variable, yields,
	\begin{equation}\label{eq:mod_cont_7}
		\norm{u_r^{k,\bar{x}}-p_k^{x_0}}_{L^{2,a}(B_1^+)}^2\leq C \epsilon,
	\end{equation}
	for any $\bar{x}\in \mathcal{Z}_k(u)$ such that $\abs{\bar{x}-x_0}<\delta_\epsilon$, and for any $r<r_\epsilon$. By a change of variable in \eqref{eq:mod_cont_1} we obtain that
	\[
	\int_{S_r^+}y^a(p_k^{\bar{x}}-p_k^{x_0})^2\ds\leq C\epsilon r^{N+a+2k}.
	\]
	Again, integrating in $r$ and with a backwards change of variable, we can infer
	\begin{equation*}
		\norm{p_k^{\bar{x}}-p_k^{x_0}}_{L^{2,a}(B_1^+)}^2\leq C\epsilon
	\end{equation*}
	for any $\bar{x}\in \mathcal{Z}_k(u)$ such that $\abs{\bar{x}-x_0}<\delta_\epsilon$, and for any $r<r_\epsilon$. Combining the previous estimate with \eqref{eq:mod_cont_7} we have that
	\begin{equation}\label{eq:mod_cont_8}
		\norm{u_r^{k,\bar{x}}-p_k^{\bar{x}}}_{L^{2,a}(B_1^+)}\leq C\sqrt{\epsilon}
	\end{equation}
	for any $\bar{x}\in \mathcal{Z}_k(u)$ such that $\abs{\bar{x}-x_0}<\delta_\epsilon$, and for any $r<r_\epsilon$. We now observe that the function $w_r^{k,\bar{x}}:=u_r^{k,\bar{x}}-p_k^{\bar{x}}\in H^{1,a}(B_1^+)$ weakly solves
	\[
	\begin{bvp}
		-\dive(y^a\nabla w_r^{k,\bar{x}})&=0, &&\text{in }B_1^+, \\
		-\lim_{y\to 0^+}y^a\frac{\partial w_r^{k,\bar{x}}}{\partial y}&= r^{k(p-2)+1-a}(\lambda_-((u_r^{k,\bar{x}})^-)^{p-1}-\lambda_+((u_r^{k,\bar{x}})^+)^{p-1}), &&\text{on }B_1'.
	\end{bvp}
	\]
	If we let
	\[
	d:=r^{k(p-2)+1-a}(\lambda_-((u_r^{k,\bar{x}})^-)^{p-1}-\lambda_+((u_r^{k,\bar{x}})^+)^{p-1}),
	\]
	in view of Corollary \ref{cor:growth_estim} we have that
	\begin{equation*}
		\norm{d}_{L^\infty(B_1')}\leq C r^{k(p-2)+1-a}.
	\end{equation*}
	Therefore, from Lemma \ref{lemma:tan2011}, by virtue of the previous inequality and \eqref{eq:mod_cont_8}, we can say that
	\begin{align*}
		\norm{w_r^{k,\bar{x}}}_{L^\infty(B_{1/2}^+)}&\leq C\left(\norm{w_r^{k,\bar{x}}}_{L^{2,a}(B_1^+)}+\norm{d}_{L^\infty(B_1')}\right), \\
		&\leq C(\sqrt{\epsilon}+r^{k(p-2)+1-a})
	\end{align*}
	for any $\bar{x}\in \mathcal{Z}_k(u)$ such that $\abs{\bar{x}-x_0}<\delta_\epsilon$, and for any $r<r_\epsilon$. In particular we can choose $r_\epsilon$ sufficiently small (independently of $\bar{x}$) in such a way that
	\[
	\norm{u_r^{k,\bar{x}}-p_k^{\bar{x}}}_{L^{\infty}(B_{1/2}^+)}\leq C\sqrt{\epsilon}.
	\]
	The conclusion of the proof follows by a covering argument of the compact set $K$.
\end{proof}

\section{Regularity of the free boundary}\label{sec:free_boundary}

The goal of this last section is to prove the results regarding the structure of the free boundary, more precisely its rectifiability and estimates on the Hausdorff dimension of its parts. A first step is the proof of the regularity of the regular part $\mathcal{R}(u)$.

\begin{proposition}\label{prop:F_sigma}
	The regular part of the free boundary $\mathcal{R}(u)$ is relatively open in $\mathcal{Z}(u)$ while, for $k\geq 2$, the $k$-singular part $\mathcal{Z}_k(u)$ is of type $F_\sigma$, i.e. the union of countably many closed sets.
\end{proposition}
\begin{proof}
	The proof of the first claim immediately follows from the upper-semicontinuity of the map $x_0\mapsto \mathcal{N}(x_0,u,0^+)$, combined with the fact that the frequency function takes no values in the range $(1,2)$. In view of Proposition \ref{prop:nondegeneracy}, the proof of the second claim is the same as \cite[Lemma 1.5.3]{Garofalo2009}, so we omit it.
\end{proof}

\begin{proof}[Proof of Proposition \ref{prop:reg_part}]
	First of all, we recall that any $a$-harmonic, $1$-homogeneous polynomial $p_1(x,y)$ must be smooth (see e.g. \cite{Sire2019}) and so it must depend only on the variable $x$ and be harmonic in $x$. In particular $p_1(x,y)=x\cdot \nnu$ for some $\nnu\in\mathbb{S}^{N-1}=\mathbb{S}^N\cap \{y=0\}$. Then, from Theorem \ref{thm:blow_up_homo} we know that, for any $x_0\in \mathcal{R}(u)$ there exists $\nnu_{x_0}\in\mathbb{S}^{N-1}$ such that
	\[
		u(x,0)=(x-x_0)\cdot\nnu_{x_0}+o(\abs{x-x_0})
	\]
	as $x\to x_0$. Moreover, we deduce from Theorem \ref{thm:cont_blow_up}  that the map $x_0\mapsto \nnu_{x_0}$ is continuous on $\mathcal{R}(u)$. Since $u(\cdot,0)\in C^{1,\alpha}(B_1')$ we can compute the directional derivative and observe that $\nabla_{x}u(x_0,0)=\nnu_{x_0}$, thus proving that $\mathcal{R}(u)\sub\{x_0\in \mathcal{Z}(u)\colon \abs{\nabla_x u(x_0,0)}\neq 0 \}$. The opposite inclusion trivially comes from Theorem \ref{thm:blow_up_homo}. Finally, the implicit function theorem allows us to conclude the proof.
\end{proof}

The following result is a key step in the proof of Theorem \ref{thm:stratification}, and we refer to \cite[Theorem 7.7]{Sire2020} for the proof, which basically relies on Theorem \ref{thm:cont_blow_up}, Proposition \ref{prop:F_sigma}, Whitney's extension theorem (as stated in \cite{Whitney1934}) and the implicit function theorem.
\begin{lemma}\label{lemma:stratification_1}
	For any $k\geq 2$ and for any $n=0,\dots,N-1$, the sets $\mathcal{Z}_k^n(u)$ are contained in a countable union of $n$-dimensional $C^1$ manifolds.
\end{lemma}

As a consequence, we are now able to prove Theorem \ref{thm:stratification}.

\begin{proof}[Proof of Theorem \ref{thm:stratification}]
	For any fixed $n\geq 0$, the sets $\mathcal{Z}_{k_1}^y(u)\cap\mathcal{Z}_{k_1}^n(u)$ and $\mathcal{Z}_{k_2}^y(u)\cap\mathcal{Z}_{k_2}^n(u)$ are disjoint for $k_1\neq k_2$; therefore
	\[
	\mathcal{S}_n^y(u)=S^y(u)\cap \big(\bigcup_{k\geq 2}\mathcal{Z}_k^n(u)  \big)
	\]
	and so it is contained in a countable union of $n$-dimensional $C^1$ manifold, thanks to Lemma \ref{lemma:stratification_1}. Then, taking the union for $n=0,\dots,N-1$, we obtain the desired result for $\mathcal{S}^y(u)$. The same applies for the stratum $\mathcal{S}^*(u)$, but in this case the dimension at any point cannot exceed the threshold $N-2$, since the polynomials in $\mathds{P}_k^*$ are harmonic in $\R^N$.
\end{proof}

The last part of this section is dedicated to the proof of Theorem \ref{thm:hausd}, regarding the Hausdorff dimension of the nodal set $\mathcal{Z}(u)$, and of its regular and singular part. The main tool we exploit is a generalization of the Federer's Reduction Principle. The basic idea is the following. Consider a family of functions $\mathcal{F}$ defined on $\R^N$ invariant under rescalings and translations, and a map $\Sigma$ which associates to every function in $\mathcal{F}$ a subset of $\R^N$. This principles establishes that, under suitable conditions on $\mathcal{F}$ and $\Sigma$, in order to control the Hausdorff dimension of $\Sigma(f)$ for every $f\in\mathcal{F}$, you just need to control the Hausdorff dimension of $\Sigma(f)$ for the functions $f$ that are homogeneous of some degree. To the best of our knowledge, this principle (originally proved by Federer) appears for the first time in the form we need in \cite[Appendix A]{Simon1983}. We make use of the following version of the principle, which is a particular case of the generalization made by Chen, see \cite[Theorem 8.5]{Chen1998a} and \cite[Proposition 4.5]{Chen1998b}.

\begin{theorem}[Federer's reduction principle]\label{thm:federer}
	Let $\mathcal{F}\sub C^{1,\alpha}_{\textup{loc}}(\R^N)$ and let, for any $u\in \mathcal{F}$, $x_0\in \R^N$ and $r>0$
	\[
	u_{r}^{x_0}(x):=u(x_0+rx).
	\]
	We say that $u_n\to u$ in $\mathcal{F}$ as $n\to\infty$ if $u_n\to u$ in $C^{1,\alpha}_{\textup{loc}}(\R^N)$. Let us assume that the family $\mathcal{F}$ satisfies the following conditions:
	\begin{itemize}
		\item[(F1)] \textnormal{(closure under appropriate translations and rescalings)} for any $r,\rho\in(0,1)$, $x_0\in B_{1-r}'$ and $u\in \mathcal{F}$ we have that $\rho u_{r}^{x_0}\in \mathcal{F}$;
		\item[(F2)] \textnormal{(existence of a homogeneous blow-up limit)} for any $x_0\in B_1'$, $r_n\to 0^+$ and $u\in \mathcal{F}$, there exists a sequence $\rho_n>0$, a real number $\alpha \geq 0$ and an $\alpha$-homogeneous function $\hat{u}\in \mathcal{F}$ such that, up to a subsequence $\rho_n u_{r_n}^{x_0}\to \hat{u}$ in $\mathcal{F}$;
		\item[(F3)] \textnormal{(singular set hypotheses)} there exists a map $\Sigma\colon \mathcal{F}\to \mathcal{C}$, where
		\[
		\mathcal{C}:=\{A\sub\R^N\colon A\cap B_1'~\text{is relatively closed in }B_1'  \},
		\]
		such that
		\begin{itemize}
			\item [(i)] for any $r,\rho \in(0,1)$, $x_0\in B_{1-r}'$ and $u\in\mathcal{F}$ we have that
			\[
			\Sigma(\rho u_{r}^{x_0})=\frac{\Sigma(u)-x_0}{r};
			\]
			\item[(ii)] for any $x_0\in B_1'$, $r_n\to 0^+$, $u,\hat{u}\in\mathcal{F}$ such that there exists $\rho_n>0$ for which, up to a subsequence, $\rho_n u_{r_n}^{x_0}\to \hat{u}$ in $\mathcal{F}$, the following holds true:
			\begin{gather*}
			\text{for any }\epsilon>0~\text{there exists }n_\epsilon>0~\text{such that } \\
			\Sigma(\rho_n u_{r_n}^{x_0})\sub \{ x\in B_1' \colon \mathrm{dist}(x,\Sigma(\hat{u}))\leq \epsilon \}\quad\text{for all }n\geq n_\epsilon.
			\end{gather*}
		\end{itemize}
	\end{itemize}
	Then either $\Sigma(u)=\emptyset$ for every $u\in\mathcal{F}$ or $\dim_H (\Sigma(u))\leq d$ for every $u\in\mathcal{F}$, where
	\begin{align*}
	d:=\max\{ \dim V\colon &V~\text{is a subspace of }\R^N~\text{and there exists }\mu\geq 0~\text{and }u\in\mathcal{F} \\
	&\text{such that }\Sigma(u)\neq \emptyset~\text{and }u_{r}^{x_0}=r^\mu u~\text{for all }x_0\in V,~r>0 \},
	\end{align*}
	Furthermore, in the latter case there exists a function $\psi\in\mathcal{F}$, a $d$-dimensional subspace $V\sub \R^N$ and a real number $\mu\geq 0$ such that
	\[
	\psi^b_r=r^\mu \psi~\text{for all }b\in V,~r>0~\text{and }\Sigma(\psi)=V\cap B_1.
	\]
	Finally, if $d=0$ the set $\Sigma(u)\cap B_r'$ is discrete for every $u\in \mathcal{F}$ and $r\in(0,1)$.
\end{theorem}

Let us consider the following class of problems, to be intended in a weak sense
\begin{equation}\label{eq:class_pbm}
	\begin{bvp}
	-\dive(y^a\nabla u)&=0, &&\text{in }B_R^+, \\
	-\lim_{y\to 0^+}y^a\frac{\partial u}{\partial y}&=\lambda\left( \lambda_-(u^-)^{p-1}-\lambda_+(u^+)^{p-1} \right), &&\text{on }B_R',
	\end{bvp}
\end{equation}
for $R\geq 1$, $\lambda\geq 0$ and let us introduce the following family of functions:
\begin{multline}\label{eq:family_F}
	\mathcal{F}:=\bigg\{ u(\cdot,0) \colon u\in C^{0,\alpha}_{\textup{loc}}(\R^{N+1}_+),~\nabla_x u\in C^{0,\alpha}_{\textup{loc}}(\R^N),~u(\cdot,0)\not\equiv 0~\text{in }\R^N\\
	\text{and solves \eqref{eq:class_pbm} for some }R\geq 1~\text{and }\lambda\geq 0  \bigg\}.
\end{multline}

\begin{remark}\label{rmk:F_homog}
	We observe that if $u\in \mathcal{F}$ is homogeneous of degree $\mu> 0$, then $\lambda=0$, because the function $t\mapsto \lambda_-(t^-)^{p-1}-\lambda_+(t^+)^{p-1}$ is homogeneous of degree $p-1$ with $p\geq 2$. This means that the even reflection of such $u$ is harmonic in $B_1$.
\end{remark}

Now we are ready to prove the last main result of our paper.

\begin{proof}[Proof of Theorem \ref{thm:hausd}]
Let $\mathcal{F}$ be as in \eqref{eq:family_F}. It is straightforward to check that assumptions (F1) and (F2) in Theorem \ref{thm:federer} are satisfied by $\mathcal{F}$, in view also of Theorem \ref{thm:blow_up_homo}. Then, we choose the map $\Sigma$ depending on the claim.

\noindent\textbf{Hausdorff dimension of $\mathcal{Z}(u)$.} In order to prove \eqref{eq:hausd_th_1}, we let $\Sigma(u):=\mathcal{Z}(u)$. It's easy to prove that the singular set hypotheses in (F3) are fulfilled by this choice of the map $\Sigma$. Therefore Theorem \ref{thm:federer} applies and we have that either $\mathcal{Z}(u)=\emptyset$ or $\dim_H(\mathcal{Z}(u))\leq d$. Assume by contradiction that $d=N$. Then there exists a function $\hat{u}\in \mathcal{F}$, homogeneous of some degree $\mu>0$, that, in view of Remark \ref{rmk:F_homog}, weakly solves
\[
	\begin{bvp}
	-\dive(y^a\nabla\hat{u})&=0, &&\text{in }B_1^+, \\
	\hat{u}&=0, &&\text{on }B_1', \\
	-\lim_{y\to 0^+}y^a\frac{\partial \hat{u}}{\partial y}&=0, &&\text{on }B_1'.
	\end{bvp}
\]
We notice that it must be $\hat{u}\equiv 0$, because otherwise its trivial extension to the whole $B_1\sub\R^{N+1}$ would violate the unique continuation principle. But this is a contradiction, since $0\notin \mathcal{F}$.

\noindent\textbf{Hausdorff dimension of $\mathcal{R}(u)$.} \eqref{eq:hausd_th_2} is an immediate consequence of Proposition \ref{prop:reg_part}.

\noindent\textbf{Hausdorff dimension of $\mathcal{S}(u)$.} Finally, we let $\Sigma(u):=\mathcal{S}(u)$. Again, it's not hard to verify the hypotheses in (F3), hence we can apply the theorem in this case as well. Moreover, since $\mathcal{S}(u)\sub \mathcal{Z}(u)$, we have that $\dim_H(\mathcal{S}(u))\leq N-1$. In fact, this bound is optimal. In order to see this, we consider \cite[Proposition 4.13]{Sire2020}. There, the authors found explicit expressions for $2$-dimensional $a$-harmonic polynomials, homogeneous of any possible integer degree. In particular, for $(s,t)\in\R^2$ and $k\in 2\N$, the following is a $k$-homogeneous, $a$-harmonic polynomial of degree $k$, even in the variable $t$
\[
p(s,t)=\frac{(-1)^{\frac{k}{2}}\Gamma\left(\frac{1}{2}+\frac{a}{2}\right)}{2^k\Gamma\left(1+\frac{k}{2}\right)\Gamma\left(\frac{1}{2}+\frac{a}{2}+\frac{k}{2}\right)}\, {}_2 F_1\left( -\frac{k}{2},-\frac{k}{2}-\frac{a}{2}+\frac{1}{2},\frac{1}{2},-\frac{s^2}{t^2} \right)t^k,
\]
where $\Gamma$ is the usual Gamma function and ${}_2 F_1$ is the hypergeometric function. To conclude the proof of \eqref{eq:hausd_th_3}, it is sufficient to consider $V=\{x\in\R^N\colon x_N=0 \}$ and $\hat{u}(x,y)=p(x_N,y)$.

\end{proof}

\section*{Acknowledgments}
\noindent
We sincerely thanks the anonymous referee for the careful reading of our manuscript and for the useful suggestions, which helped to improve our work. R. Ognibene acknowledges support from the MIUR-PRIN project No. 2017TEXA3H. R. Ognibene is member of the GNAMPA group of the Istituto Nazionale di Alta Matematica (INdAM). This work started during a visit of R. Ognibene at Purdue University (Department of Mathematics), which he warmly thanks for the hospitality.

\bibliography{biblio}

\begin{thebibliography}{10}

\bibitem{Abatangelo20191}
{\sc Abatangelo, N., and Valdinoci, E.}
\newblock Getting acquainted with the fractional laplacian.
\newblock {\em Springer INdAM Series 33\/} (2019), 1--105.

\bibitem{Allen2012b}
{\sc Allen, M.}
\newblock Separation of a lower dimensional free boundary in a two-phase
  problem.
\newblock {\em Math. Res. Lett. 19}, 5 (2012), 1055--1074.

\bibitem{Allen2019}
{\sc Allen, M.}
\newblock A fractional free boundary problem related to a plasma problem.
\newblock {\em Communications in Analysis and Geometry 27}, 8 (2019),
  1665--1696.

\bibitem{Allen2015}
{\sc Allen, M., Lindgren, E., and Petrosyan, A.}
\newblock The two-phase fractional obstacle problem.
\newblock {\em SIAM J. Math. Anal. 47}, 3 (2015), 1879--1905.

\bibitem{Allen2012a}
{\sc Allen, M., and Petrosyan, A.}
\newblock A two-phase problem with a lower-dimensional free boundary.
\newblock {\em Interfaces Free Bound. 14}, 3 (2012), 307--342.

\bibitem{AC2004}
{\sc Athanasopoulos, I., and Caffarelli, L.~A.}
\newblock Optimal regularity of lower dimensional obstacle problems.
\newblock {\em Zap. Nauchn. Sem. S.-Peterburg. Otdel. Mat. Inst. Steklov.
  (POMI) 310}, Kraev. Zadachi Mat. Fiz. i Smezh. Vopr. Teor. Funkts. 35 [34]
  (2004), 49--66, 226.

\bibitem{ACS2008}
{\sc Athanasopoulos, I., Caffarelli, L.~A., and Salsa, S.}
\newblock The structure of the free boundary for lower dimensional obstacle
  problems.
\newblock {\em Amer. J. Math. 130}, 2 (2008), 485--498.

\bibitem{Bucur2016}
{\sc Bucur, C., and Valdinoci, E.}
\newblock {\em Nonlocal diffusion and applications}, vol.~20 of {\em Lecture
  Notes of the Unione Matematica Italiana}.
\newblock Springer, [Cham]; Unione Matematica Italiana, Bologna, 2016.

\bibitem{Cabre2014}
{\sc Cabr\'{e}, X., and Sire, Y.}
\newblock Nonlinear equations for fractional {L}aplacians, {I}: {R}egularity,
  maximum principles, and {H}amiltonian estimates.
\newblock {\em Ann. Inst. H. Poincar\'{e} Anal. Non Lin\'{e}aire 31}, 1 (2014),
  23--53.

\bibitem{Caffarelli2007}
{\sc Caffarelli, L., and Silvestre, L.}
\newblock An extension problem related to the fractional {L}aplacian.
\newblock {\em Comm. Partial Differential Equations 32}, 7-9 (2007),
  1245--1260.

\bibitem{Caffarelli2008}
{\sc Caffarelli, L.~A., Salsa, S., and Silvestre, L.}
\newblock Regularity estimates for the solution and the free boundary of the
  obstacle problem for the fractional {L}aplacian.
\newblock {\em Invent. Math. 171}, 2 (2008), 425--461.

\bibitem{Chen2020}
{\sc Chen, W., Li, Y., and Ma, P.}
\newblock {\em The fractional {L}aplacian}.
\newblock World Scientific Publishing Co. Pte. Ltd., Hackensack, NJ, [2020]
  \copyright 2020.

\bibitem{Chen1998b}
{\sc Chen, X.-Y.}
\newblock On the scaling limits at zeros of solutions of parabolic equations.
\newblock {\em J. Differential Equations 147}, 2 (1998), 355--382.

\bibitem{Chen1998a}
{\sc Chen, X.-Y.}
\newblock A strong unique continuation theorem for parabolic equations.
\newblock {\em Math. Ann. 311}, 4 (1998), 603--630.

\bibitem{Colombo2020}
{\sc Colombo, M., Spolaor, L., and Velichkov, B.}
\newblock Direct epiperimetric inequalities for the thin obstacle problem and
  applications.
\newblock {\em Comm. Pure Appl. Math. 73}, 2 (2020), 384--420.

\bibitem{Danielli2021}
{\sc Danielli, D., and Jain, R.}
\newblock Regularity results for a penalized boundary obstacle problem.
\newblock {\em Math. Eng. 3}, 1 (2021), Paper No. 7, 23.

\bibitem{Danielli2018}
{\sc Danielli, D., and Salsa, S.}
\newblock Obstacle problems involving the fractional {L}aplacian.
\newblock In {\em Recent developments in nonlocal theory}. De Gruyter, Berlin,
  2018, pp.~81--164.

\bibitem{DSV2021}
{\sc De~Philippis, G., Spolaor, L., and Velichkov, B.}
\newblock Regularity of the free boundary for the two-phase {B}ernoulli
  problem.
\newblock {\em Invent. Math. 225}, 2 (2021), 347--394.

\bibitem{Dinezza2012}
{\sc Di~Nezza, E., Palatucci, G., and Valdinoci, E.}
\newblock Hitchhiker's guide to the fractional {S}obolev spaces.
\newblock {\em Bull. Sci. Math. 136}, 5 (2012), 521--573.

\bibitem{Fabes1982}
{\sc Fabes, E.~B., Kenig, C.~E., and Serapioni, R.~P.}
\newblock The local regularity of solutions of degenerate elliptic equations.
\newblock {\em Comm. Partial Differential Equations 7}, 1 (1982), 77--116.

\bibitem{Fall2014}
{\sc Fall, M.~M., and Felli, V.}
\newblock Unique continuation property and local asymptotics of solutions to
  fractional elliptic equations.
\newblock {\em Comm. Partial Differential Equations 39}, 2 (2014), 354--397.

\bibitem{Figalli2020}
{\sc Figalli, A., Ros-Oton, X., and Serra, J.}
\newblock Generic regularity of free boundaries for the obstacle problem.
\newblock {\em Publ. Math. Inst. Hautes \'{E}tudes Sci. 132\/} (2020),
  181--292.

\bibitem{Figalli2019}
{\sc Figalli, A., and Serra, J.}
\newblock On the fine structure of the free boundary for the classical obstacle
  problem.
\newblock {\em Invent. Math. 215}, 1 (2019), 311--366.

\bibitem{Fotouhi2017}
{\sc Fotouhi, M., and Shahgholian, H.}
\newblock A semilinear {PDE} with free boundary.
\newblock {\em Nonlinear Anal. 151\/} (2017), 145--163.

\bibitem{MR3916700}
{\sc Garofalo, N.}
\newblock Fractional thoughts.
\newblock In {\em New developments in the analysis of nonlocal operators},
  vol.~723 of {\em Contemp. Math.} Amer. Math. Soc., Providence, RI, 2019,
  pp.~1--135.

\bibitem{Garofalo2009}
{\sc Garofalo, N., and Petrosyan, A.}
\newblock Some new monotonicity formulas and the singular set in the lower
  dimensional obstacle problem.
\newblock {\em Invent. Math. 177}, 2 (2009), 415--461.

\bibitem{MR3491531}
{\sc Garofalo, N., Petrosyan, A., and Smit Vega~Garcia, M.}
\newblock An epiperimetric inequality approach to the regularity of the free
  boundary in the {S}ignorini problem with variable coefficients.
\newblock {\em J. Math. Pures Appl. (9) 105}, 6 (2016), 745--787.

\bibitem{GRO2019}
{\sc Garofalo, N., and Ros-Oton, X.}
\newblock Structure and regularity of the singular set in the obstacle problem
  for the fractional {L}aplacian.
\newblock {\em Rev. Mat. Iberoam. 35}, 5 (2019), 1309--1365.

\bibitem{Murthy1968}
{\sc Murthy, M. K.~V., and Stampacchia, G.}
\newblock Boundary value problems for some degenerate-elliptic operators.
\newblock {\em Ann. Mat. Pura Appl. (4) 80\/} (1968), 1--122.

\bibitem{Nekvinda1993}
{\sc Nekvinda, A.}
\newblock Characterization of traces of the weighted {S}obolev space
  {$W^{1,p}(\Omega,d^\epsilon_M)$} on {$M$}.
\newblock {\em Czechoslovak Math. J. 43(118)}, 4 (1993), 695--711.

\bibitem{PSU2012}
{\sc Petrosyan, A., Shahgholian, H., and Uraltseva, N.}
\newblock {\em Regularity of free boundaries in obstacle-type problems},
  vol.~136 of {\em Graduate Studies in Mathematics}.
\newblock American Mathematical Society, Providence, RI, 2012.

\bibitem{Ruland2017}
{\sc R\"{u}land, A.}
\newblock On quantitative unique continuation properties of fractional
  {S}chr\"{o}dinger equations: doubling, vanishing order and nodal domain
  estimates.
\newblock {\em Trans. Amer. Math. Soc. 369}, 4 (2017), 2311--2362.

\bibitem{Ruland2018}
{\sc R\"{u}land, A.}
\newblock Unique continuation for sublinear elliptic equations based on
  {C}arleman estimates.
\newblock {\em J. Differential Equations 265}, 11 (2018), 6009--6035.

\bibitem{Simon1980a}
{\sc Schechter, M., and Simon, B.}
\newblock Unique continuation for {S}chr\"{o}dinger operators with unbounded
  potentials.
\newblock {\em J. Math. Anal. Appl. 77}, 2 (1980), 482--492.

\bibitem{Signorini1959}
{\sc Signorini, A.}
\newblock Questioni di elasticit\`a non linearizzata e semilinearizzata.
\newblock {\em Rend. Mat. e Appl. (5) 18\/} (1959), 95--139.

\bibitem{Simon1983}
{\sc Simon, L.}
\newblock {\em Lectures on geometric measure theory}, vol.~3 of {\em
  Proceedings of the Centre for Mathematical Analysis, Australian National
  University}.
\newblock Australian National University, Centre for Mathematical Analysis,
  Canberra, 1983.

\bibitem{Sire2020}
{\sc Sire, Y., Terracini, S., and Tortone, G.}
\newblock On the nodal set of solutions to degenerate or singular elliptic
  equations with an application to s-harmonic functions.
\newblock {\em Journal des Mathematiques Pures et Appliquees\/} (2020).

\bibitem{Sire2019}
{\sc Sire, Y., Terracini, S., and Vita, S.}
\newblock Liouville type theorems and regularity of solutions to degenerate or
  singular problems part i: even solutions.
\newblock {\em Communications in Partial Differential Equations 46}, 2 (2021),
  310--361.

\bibitem{Soave2018}
{\sc Soave, N., and Terracini, S.}
\newblock The nodal set of solutions to some elliptic problems: sublinear
  equations, and unstable two-phase membrane problem.
\newblock {\em Adv. Math. 334\/} (2018), 243--299.

\bibitem{Soave2019}
{\sc Soave, N., and Terracini, S.}
\newblock The nodal set of solutions to some elliptic problems: singular
  nonlinearities.
\newblock {\em J. Math. Pures Appl. (9) 128\/} (2019), 264--296.

\bibitem{Spolaor2021}
{\sc Spolaor, L., and Velichkov, B.}
\newblock On the logarithmic epiperimetric inequality for the obstacle problem.
\newblock {\em Math. Eng. 3}, 1 (2021), Paper No. 4, 42.

\bibitem{Tan2011}
{\sc Tan, J., and Xiong, J.}
\newblock A {H}arnack inequality for fractional {L}aplace equations with lower
  order terms.
\newblock {\em Discrete Contin. Dyn. Syst. 31}, 3 (2011), 975--983.

\bibitem{Tao2008}
{\sc Tao, X., and Zhang, S.}
\newblock Weighted doubling properties and unique continuation theorems for the
  degenerate {S}chr\"{o}dinger equations with singular potentials.
\newblock {\em J. Math. Anal. Appl. 339}, 1 (2008), 70--84.

\bibitem{Tortone2020}
{\sc Tortone, G.}
\newblock {The nodal set of solutions to some nonlocal sublinear problems}.
\newblock {\em Preprint\/} (2020).

\bibitem{Wang2016}
{\sc Wang, K., and Wei, J.}
\newblock On the uniqueness of solutions of a nonlocal elliptic system.
\newblock {\em Math. Ann. 365}, 1-2 (2016), 105--153.

\bibitem{Whitney1934}
{\sc Whitney, H.}
\newblock Analytic extensions of differentiable functions defined in closed
  sets.
\newblock {\em Trans. Amer. Math. Soc. 36}, 1 (1934), 63--89.

\end{thebibliography}
\bibliographystyle{acm}

\end{document}